\documentclass[11pt]{amsart}

\usepackage{amsmath, amssymb ,amsthm, amsfonts, amsgen,color}
\usepackage{bbm}
\usepackage{tikz}

\numberwithin{equation}{section}

\newcommand{\R}{{\mathbb{R}}}

\newcommand{\beq}{\begin{equation}}
\newcommand{\eeq}{\end{equation}}

\newcommand{\weaklystar}{\rightharpoonup^\ast}
\newcommand{\weakly}{\rightharpoonup}

\def\XXint#1#2#3{{\setbox 0=\hbox{$#1{#2#3}{\int}$}
\vcenter{\hbox{$#2#3$}}\kern-.5\wd0}}

\makeatletter
\def\rightharpoonupfill@{\arrowfill@\relbar\relbar\rightharpoonup}

\newcommand{\xrightharpoonup}[2][]{\ext@arrow
0359\rightharpoonupfill@{#1}{#2}} \makeatother

\newtheorem{Theorem}{Theorem}[section]
\newtheorem{Lemma}[Theorem]{Lemma}
\newtheorem{Proposition}[Theorem]{Proposition}
\newtheorem{Corollary}[Theorem]{Corollary}

\newtheorem{Remark}[Theorem]{Remark}
\newtheorem{Definition}[Theorem]{Definition}
\newtheorem{Example}[Theorem]{Example}

\def\supess{\mathop{\rm ess\: sup }}

\title[Nonlocal supremal functionals ]{Lower semicontinuity and relaxation of \\ nonlocal \boldmath{$L^\infty$}-functionals}

\author{Carolin Kreisbeck}
\address{Mathematisch Instituut, Universiteit Utrecht, Postbus 80010, 3508 TA Utrecht, The Netherlands}
\email{c.kreisbeck@uu.nl}

\author{Elvira Zappale}
\address{D.I.In., Universit\`a degli Studi di Salerno, Via Giovanni Paolo II 132, 84084 Fisciano, SA, Italy}
\email{ezappale@unisa.it }

\begin{document} 

\maketitle
 \begin{abstract} 
We study variational problems involving nonlocal supremal functionals
\begin{align*}
L^\infty(\Omega;\R^m) \ni u\mapsto {\rm ess sup}_{(x,y)\in \Omega\times \Omega} W(u(x), u(y)), 
\end{align*}
where $\Omega\subset \mathbb{R}^n$ is a bounded, open set and $W:\mathbb{R}^m\times\mathbb{R}^m\to \mathbb{R}$ is a suitable function. Motivated by existence theory via the direct method, we identify a necessary and sufficient condition for $L^\infty$-weak$^\ast$ lower semicontinuity of these functionals, namely, separate level convexity 
of a symmetrized and suitably diagonalized version of the supremands. 
More generally, we show that the supremal structure of the functionals is preserved during the process of relaxation. The analogous statement in the related context of double-integral functionals was recently shown to be false. 
Our proof relies substantially on the connection between supremal and indicator functionals. This allows us to recast the relaxation problem into characterizing weak$^\ast$ closures of a class of nonlocal inclusions, which is of independent interest.  
To illustrate the theory, we determine explicit relaxation formulas for examples of functionals with different multi-well supremands.      
\vspace{8pt}

 
 
 \vspace{8pt}
 
 \noindent\textsc{Date:} \today
 \end{abstract}

	\section{Introduction}	\footnote{\color{olive}The colored parts indicate minor corrections and improvements to the version published in Calc. Var. 59:138 (2020), doi:10.1007/s00526-020-01782-w.}      
	Nonlocal functionals in the form of double integrals appear naturally in different applications; examples include peridynamics~\cite{BMP15, MeD15,Sil00}, image processing~\cite{BrN18, GiO08} or the theory of phase transitions~\cite{CDV17, DFL18, OvV12}. 
In the homogeneous case, separate convexity of the integrands has been identified as a necessary and sufficient condition for the weak lower semicontinuity of such functionals~\cite{BeP06, Mun09, Ped97}. 
When it comes to relaxation, meaning the characterization of weak lower semicontinuous envelopes, though, the problem is still largely open. The difficulty lies in the fact that, counterintuitively,  
relaxation formulas in general cannot be obtained via separate convexification of the integrands, as explicit examples in~\cite{BMC18, BeP06,  Ped16} indicate. 
As first shown in~\cite{KrZ19}, and with different techniques in ~\cite{MCT20}, even a representation of the relaxation with a double integral of the same type is not always possible. 

Inspired by these recent developments, as well as new models arising in the theory of machine learning (see e.g.~\cite{FCAO18}), this article addresses a related problem by discussing homogeneous supremal (or $L^\infty$-)functionals in the nonlocal setting, i.e.,
\begin{align}\label{ourfunct}
L^\infty(\Omega;\mathbb{R}^m) \ni u\mapsto J(u):= {\rm ess sup}_{(x,y)\in \Omega\times \Omega} W(u(x), u(y)), 
\end{align}
where $\Omega\subset \mathbb{R}^n$ is a bounded, open set and $W:\mathbb{R}^m\times\mathbb{R}^m\to \mathbb{R}$ is a given Borel function satisfying suitable further assumptions regarding continuity and coercivity. We contribute answers to two key questions, which are motivated by the existence theory for solutions to variational problems in form of the direct methods in the calculus of variations: 
\begin{itemize} 
\item[(Q1)] What are necessary and sufficient conditions on the supremand $W$ for the (sequential) lower semicontinuity of $J$ with respect to the natural topology, that is, the $L^\infty$-weak$^\ast$ topology?
\item[(Q2)] 
If $J$ fails to satisfy the conditions resulting as an answer to (Q1), can we find an explicit representation of its relaxation, that is, of its $L^\infty$-weak$^\ast$ (sequential) lower semicontinuous envelope?
\end{itemize}

Notice that in the context of this paper, the $L^\infty$-weak$^*$ topology and the sequential one can always be used interchangeably, as the former admits a metrizable description on bounded sets;
see Remark~\ref{rem1.2}\,a) for a more detail. 
 
We point out that inhomogeneous versions of~\eqref{ourfunct} appeared already in~\cite{GPP06}, and more lately  in~\cite{Guo17, KSZ14}. Moreover, 
it is useful to observe that 
functionals of the type~\eqref{ourfunct} share key features with two different classes of functionals that have been studied intensively in the literature, namely double-integral functionals mentioned already at the beginning, i.e.,
\begin{align*}
L^p(\Omega;\R^m)\ni u\mapsto \int_{\Omega}\int_\Omega W(u(x), u(y))\,d{x}\, dy
\end{align*}
with $p\in [1, \infty)$, 
and supremal functionals (or $L^\infty$-functionals), i.e.,
\begin{align*}
L^\infty(\Omega;\R^m)\ni u\mapsto \supess_{x\in \Omega} f(u(x))
\end{align*} 
with a suitable function $f:\R^m\to \R$; for more details and background on these two branches of research, including a list of references, we refer to Sections~\ref{subsec:supremal} and~\ref{subsec:nonlocal}. 
Borrowing and combining methods and techniques from these two fields, which are largely based on Young measure theory, equip us with quite a rich tool box for analyzing nonlocal supremal functionals. However, it will become clear in the following that, in order to settle the questions (Q1) and (Q2), new ideas are needed in addition. 

A crucial realization is that the functional $J$ in~\eqref{ourfunct} remains unaffected by certain changes of $W$, beyond mere symmetrization. Indeed, replacing $W$ with its diagonalized and symmetrized version $\widehat W$ (see~\eqref{subsec:lsc_relaxation} along with Section~\ref{sec:diagonal} for the precise definition) still gives the same functional.

To understand better the role of diagonalization, it helps to take a different perspective on our nonlocal supremal functionals and to exploit their connection to the so-called nonlocal indicator functionals. These are double integrals over the characteristic function $\chi_K$ for a compact set $K\subset \R^m\times \R^m$, i.e.,  
\begin{align}\label{indicator_intro}
L^\infty(\Omega;\R^m)\ni u\mapsto \int_\Omega\int_\Omega \chi_K(u(x), u(y))\,dx \,dy.
\end{align} 
By modification of a result due to Barron, Jensen \& Wang~\cite[Lemma~1.4]{BJW01}, we find that (Q1) and (Q2) for $J$ in~\eqref{ourfunct} are equivalent to studying the same questions for all indicator functionals associated with the sublevel sets of $W$, cf.~Proposition~\ref{propequivsupremalunbounded}. 
Then again,~\eqref{indicator_intro} is closely tied to nonlocal inclusions of the form
\begin{align}\label{nonlocal_inclusion}
(u(x), u(y))\in K &\text{ for a.e.~$(x,y)\in \Omega\times \Omega$,}
\end{align} 
and (Q2) comes down to identifying the asymptotic behavior of $L^\infty$-weakly$^\ast$ converging sequences subject to this type of constraint, which is also of independent interest. If we denote by $\mathcal A_K$ the set of all functions in $L^\infty(\Omega;\R^m)$ satisfying~\eqref{nonlocal_inclusion}, the task is to  characterize the $L^\infty$-weak$^\ast$ closure of $\mathcal A_K$.
In the classical local setting, that is, when~\eqref{nonlocal_inclusion} is changed into~
\begin{align}
	\label{classical}
u(x)\in A \text{ for a.e.~$x\in \Omega$ with $A\subset\R^m$ compact},
\end{align}
it is well known that the $L^\infty$-weak$^\ast$ limits of sequences with this property correspond to essentially bounded functions with values in the convex hull of $A$. In the nonlocal case, where one expects the separate convexification to take over the role of convexification in the local problem, things turn out to be a bit more subtle.  

The reason lies in the special interaction between nonlocality and the pointwise constraint, which makes~\eqref{nonlocal_inclusion} substantially different from the classical case~\eqref{classical}, as this simple example illustrates.  If $m=1$ and $K=\{(1, 0), (-1,0), (0,1), (0,-1)\}\subset \R\times \R$, then $\mathcal A_K=\emptyset$, cf.~Example~\ref{exampleK} and~\eqref{AE=AEhat}. 
For a general compact 
$K\subset \R^m\times \R^m$, we show in Proposition~\ref{lem:Ik=ItildeK} that the nonlocal inclusion~\eqref{nonlocal_inclusion} is invariant under symmetrization and diagonalization of $K$, i.e., 
\begin{align}\label{AK=AKhat_intro}
 \mathcal A_K=\mathcal A_{\widehat K}
 \end{align}
with
\begin{align}\label{Khat_intro}
\widehat K:=\{(\xi, \zeta)\in K: (\xi, \zeta), (\xi, \xi), (\zeta, \zeta)\in K\}.
\end{align}
Based on this observation, we prove the following characterization of $L^\infty$-weak$^\ast$ limits of sequences in $\mathcal A_K$. Particularly, this result is one of the main ingredients for answering questions (Q1) and (Q2). 

\begin{Theorem}\label{theo:weakclosure_intro}
Let $K\subset \R^m\times \R^m$ be compact, let $\widehat K$ be the symmetric and diagonal version of $K$ in the sense of~\eqref{Khat_intro}, and let $\widehat K^{\rm sc}$ be the separately convex hull of $\widehat K$, see Definition~\ref{separateconvexitysets} below. 
If $m>1$, assume in addition that $\widehat K^{\rm sc}$ is compact and  that the symmetrization and diagonalization of $\widehat K^{\rm sc}$ can be represented as the union of all cubes of the form $[\alpha, \beta]\times [\alpha, \beta]$ with $(\alpha, \beta)\in K$, \color{olive} which are supposed to comprise the maximal Cartesian subsets of $\widehat{K}^{\rm sc}$, \color{black} cf.~\eqref{ass76} \color{olive} and Definition~\ref{def:Cartesian_subset}. \color{black}

Then, the (sequential) $L^\infty$-weak$^\ast$ closure of $\mathcal A_K$ is given by 
$\mathcal A_{\widehat K^{\rm sc}}$.
\end{Theorem}

\begin{Remark}\label{rem1.2}  a) In light of the well-known fact that the $L^\infty$-weak$^\ast$ topology is metrizable on bounded sets (see e.g.~\cite[A.1.5]{FoL07}), the compactness hypothesis on $K$ in the above theorem guarantees the equivalence between the use of the $L^\infty$-weak$^\ast$ topology and the corresponding sequential version. 

b) Theorem \ref{theo:weakclosure_intro} implies that $\mathcal A_K$ is weakly$^\ast$ closed if and only if 
\begin{align}\label{KK}
\widehat{\widehat K^{\rm sc}} = \widehat K,
\end{align}
which, in the scalar case $m=1$, is equivalent to the separate convexity of $\widehat K$, cf.~Corollary~\ref{cor:Khatsc}. Notice that this necessary and sufficient condition is strictly weaker than requiring that $K$ is separately convex. \\[-0.2cm] 

c) As an immediate corollary of Theorem~\ref{theo:weakclosure_intro}, we obtain that the relaxation of the indicator functional~\eqref{indicator_intro} is given by
\begin{align*}
L^\infty(\Omega;\R^m)\ni u\mapsto \int_\Omega\int_\Omega \chi_{\widehat{K}^{\rm sc}}(u(x), u(y))\,dx \,dy;
\end{align*}
in particular,~\eqref{indicator_intro} is $L^\infty$-weak$^\ast$ lower semicontinuous if and only if~\eqref{KK} holds, cf.~Corollary~\ref{cor:chara_wlscp_indicator}.
\end{Remark}

The proof of Theorem~\ref{theo:weakclosure_intro} relies on a series of auxiliary results. With~\eqref{AK=AKhat_intro} established in Proposition~\ref{lem:Ik=ItildeK}, an argument based on pointwise approximation by piecewise affine functions allows us to deduce a refined representation of elements $\mathcal A_K$, saying that for each $u\in \mathcal A_K$ there exists a Cartesian product $A\times A\subset K$ with $A\subset \R^m$ such that $u\in \mathcal A_{A\times A}$, see Proposition~\ref{prop:3}. 
Another important ingredient in the case $m=1$ is a characterization of the separately convex hull of $\widehat K$, which can be shown to have a particularly simple form. In fact, $\widehat K^{\rm sc}$ is the union of all squares in $\R\times \R$ whose corners are extreme points (in the sense of separate convexification of) $\widehat K$, for details see~Corollary~\ref{cor:square}. 
In higher dimensions, the analogous statement, which could be viewed as a Caratheodory type formula, is in general false (cf.~Remark~\ref{rem:failure}\,c)); the required extra assumptions on $\widehat K^{\rm sc}$ if $m>1$ are introduced to  compensate for this. Combining all the previous arguments 
reduces the proof of Theorem~\ref{theo:weakclosure_intro} to the case when $K$ takes the form of a Cartesian product in $\R^m\times \R^m$. Under this assumption, the desired $L^\infty$-weak$^\ast$ approximation of $u\in \mathcal A_{K^{\rm sc}}$ follows from an explicit construction of periodically oscillating sequences, see~Lemma~\ref{lem:1}. Alternatively, one could use a more abstract approach via Young measures generated by sequences that satisfy an approximate nonlocal constraint, together with a projection step to enforce the exact nonlocal inclusion~\eqref{nonlocal_inclusion}, cf.~Proposition~\ref{prop:YMchar_Cartesian}.   

Conceptually, the study of nonlocal inclusions as in~\eqref{nonlocal_inclusion} shows close parallels with the field of differential inclusions, dealing with problems such as
\begin{align*}
\nabla u\in M\quad \text{ a.e.~in $\Omega$  and $M\subset \R^{m\times n}$ compact}
\end{align*} 
for $u\in W^{1, \infty}(\Omega;\R^m)$ (see e.g.~\cite{Dac08, Rin18} and the references therein), and compensated compactness theory~\cite{Mur78, Tar79}; notice that the latter deal with problems that are all local in nature. The overall challenge is to capture the interplay between pointwise constraints and the structural properties of the vector fields, whether they are gradients, or more generally, $\mathcal A$-free fields with some differential operator $\mathcal A$, or, like here, nonlocal vector fields of the form~\eqref{vu}. Yet, besides these conceptual parallels, nonlocality creates effects that are not typically encountered in local problems, as for instance~\eqref{AK=AKhat_intro} indicates.

In generalization of Theorem~\ref{theo:weakclosure_intro}, we characterize the set of Young measures generated by nonlocal vector fields associated with uniformly bounded sequences $(u_j)_j\subset L^\infty(\Omega;\R^m)$, cf.~\eqref{vu}; 
indeed, if $(u_j)_j$ 
generates the Young measure $\nu=\{\nu_x\}_{x\in \Omega}$, the sought-after set consists of all the product measures $\Lambda=\{\Lambda_{(x,y)}\}_{(x,y)\in \Omega\times \Omega} = \{\nu_x\otimes \nu_y\}_{(x,y)\in \Omega\times \Omega}$ with ${\rm supp\,} \Lambda$ contained almost everywhere in a Cartesian subset of $K$, see~Theorem~\ref{theo:YMcharacterization} for the precise statement. Interpreted in the context of indicator functionals, the latter yields 
a Young measure relaxation result for a class of unbounded functionals (defined precisely in~\eqref{extended}), extending part of a recent work by Bellido \& Mora-Corral~\cite[Section~6]{BMC18}, cf.~Section~\ref{sec:YMrelaxation}. 

The next theorem collects the main results of this paper regarding nonlocal supremal functionals.
 In contrast to the theory of double-integral functionals,
we show here that relaxation of nonlocal supremal functionals is structure preserving, in the sense that it is again of nonlocal supremal type. For simplicity, we formulate the result here in the scalar case; for the extension to the vectorial setting (under additional conditions), we refer to Corollary~\ref{cor} and Remark~\ref{rem:Jwlsc_m}. 

	\begin{Theorem}\label{theo:main}
Let $J$ be as in~\eqref{ourfunct} and $W:\R\times \R\to \R$ be lower semicontinuous and coercive, i.e., $W(\xi, \zeta)\to \infty$ as $|(\xi, \zeta)|\to \infty$. 
\begin{itemize}
\item[$(i)$] 
The functional $J$ is $L^\infty$-weakly$^\ast$ lower semicontinuous if and only if $\widehat{W}$ is separately level convex, where $\widehat{W}$, defined in~\eqref{defWhat}, is the density resulting from diagonalization and symmetrization of $W$. 
\item[$(ii)$] The relaxation $J^{\rm rlx}$ of $J$ is given by the nonlocal supremal functional of the form~\eqref{ourfunct} with supremand $\widehat{W}^{\rm slc}$, which is the separately level convex envelope of $\widehat W$.
\end{itemize}
	\end{Theorem}
Referring back to the beginning of the introduction, we stress the link between nonlocal supremal functionals and nonlocal double-integral functionals via $L^p$-approximation; if $W=\widehat W$ is separately level convex, this can be made rigorous by imitating the arguments by Champion, De Pascale \& Prinari in~\cite[Theorem~3.1]{CDP04}. 
	
 As an outlook on interesting future research beyond the scope of this work, we would like to mention in particular the proof of a characterization result for the $L^\infty$-weak$^\ast$ closure of $\mathcal A_K$ in general dimensions without extra assumptions on $K$, or the extension to our theory to inhomogeneous nonlocal functionals. 
		
The paper is organized as follows. First, we collect some preliminaries in Section~\ref{sec:preliminaries}; these include subsections on frequently used notation, auxiliary results for Young measures, as well as background on the theories of both supremal and nonlocal double-integral functionals. After introducing and discussing the notion of separate level convexity in Section~\ref{3}, we investigate the interaction of separate convexification of sets with their diagonalization and symmetrization in Section~\ref{sec:diagonal}. In Section~\ref{sec:nonlocal_inclusions}, we turn to the analysis of nonlocal inclusions; more precisely, Subsection~\ref{subsec:alternative} provides alternative representations of $\mathcal A_K$,
Subsection~\ref{subsec:asymptotics} contains the proof of Theorem~\ref{theo:weakclosure_intro}, and Subsection~\ref{subsec:YM_characterization} is concerned with the characterization of Young measures generated by sequences of nonlocal vector fields. 
In Section~\ref{sec:indicator}, we reformulate the insights about nonlocal inclusions in terms of nonlocal indicator functionals (see Subsections~\ref{subsec:51} and~\ref{sec:YMrelaxation}), and discuss the connection between different notions of nonlocal convexity for extended-valued functionals (see Subsection~\ref{subsec:nonlocal_convexity}). The main theorems on lower semicontinuity and relaxation of nonlocal supremal functionals, which address the questions (Q1) and (Q2), are established in Section~\ref{6}. To illustrate the theory, we finally present a few examples of nonlocal supremal functionals with different multiwell supremands in Subsection~\ref{subsec:examples}, and determine explicitly the corresponding relaxation formulas. 

	\section{Preliminaries}\label{sec:preliminaries}
			In this section, we fix notations and recall some well-known results that will be exploited in the remainder of the paper.

	\subsection{Notation}\label{not} 
	In the following, $m$ and $n$ are natural numbers. For any vector $\xi\in \mathbb R^m$, let $\xi_i$, $i=1,\dots, m$, denote its components,  
	and $|\xi|= (\sum_{i=1}^m\xi_i^2)^{\frac{1}{2}}$ its Euclidean norm. By $B_r(\xi)$, we denote the closed (Euclidean) ball centered in $\xi\in \R^m$ with radius $r>0$.
For two vectors $\alpha, \beta\in \R^m$, we introduce the generalized closed interval 
 \begin{align}\label{generalizedinterval}
 [\alpha, \beta] :=\{t\alpha + (1-t)\beta: t\in [0,1]\} \subset \R^m,
 \end{align} 
 and analogously, the open and half open segments $[\alpha, \beta[$, $]\alpha, \beta]$, and $]\alpha, \beta[$;  moreover, let us define 
 \begin{align}\label{cubesQalphabeta}
 Q_{\alpha, \beta}:=[\alpha, \beta]\times [\alpha, \beta] \subset \R^m\times \R^m.
 \end{align}

Our notation for the complement of a set $A\subset \R^{m}$ is $A^c=\R^{m}\setminus A$, whilst $A^{\rm co}$ stands for the convex hull of $A$. 
Moreover, we denote the characteristic function of $A\subset \R^m$ in the sense of convex analysis by $\chi_A$ and the indicator function of $A$ by $\mathbbm{1}_A$, i.e.~
	\begin{equation}\label{indicator}
	\chi_A (\xi):=\left\{
	\begin{array}{ll} 0 &\hbox{ if $\xi\in A$},\\
	\infty &\hbox{ otherwise,}
	\end{array}
	\right.
	\quad \text{and}\quad 	\mathbbm{1}_A (\xi):=\left\{
	\begin{array}{ll} 1 &\hbox{ if $\xi\in A$},\\
	0 &\hbox{ otherwise.}
	\end{array}
	\right.
	\end{equation}
	The distance from a point $\beta \in \mathbb R^m$ to a set $A \subset \mathbb R^m$ is 
	${\rm dist}(\beta, A):=\inf_{\alpha \in A}|\alpha-\beta|$, and the
  Hausdorff distance between two non-empty sets $A, B \subset \R^m$ is given by 
	\begin{align}\label{Hausdorff}
	d_H^m(A, B) := {\rm sup}_{\alpha\in A}\,  {\rm dist}(\alpha, B) + {\rm sup}_{\beta\in B}\, {\rm dist}(\beta, A).
	\end{align} 
 
Further, we denote by $\R_\infty$ the set $\R\cup\{\infty\}$. 
 For every $c \in \mathbb R$ and every function $f:\mathbb R^m \to \mathbb R_\infty$, 
\begin{align*}
L_c(f):=\{\xi \in \mathbb R^m: f(\xi)\leq c\} \subset \R^m
\end{align*}
is the sublevel set of $f$ at level $c$.

 Let $E\subset A\times A$ with $A \subset \mathbb R^m$; then $\pi_1(E)$ and $\pi_2(E)$ stand for the the projection of $E$ onto the first and second component, respectively, that is
	 \begin{align*}
	 \pi_1(E) = \bigcup_{(\alpha, \beta)\in E} \alpha \quad \text{and} \quad \pi_2(E)= \bigcup_{(\alpha, \beta)\in E} \beta.
	 \end{align*}
To denote the sections of $E$ in the first and second argument at $\beta\in A$, we use a notation with letters in Frakture, precisely, 
\begin{align*}
\mathfrak{E}_1^\beta :=\{\alpha\in A: (\alpha, \beta)\in E\}\quad \text{and}\quad \mathfrak{E}_2^\beta :=\{\alpha\in A: (\beta, \alpha)\in E\}.
\end{align*} 
	If $E$ is symmetric, meaning $E=E^T$ with $E^T:=\{(\alpha, \beta)\in A\times A: (\beta,\alpha)\in E\}$, then $\pi_1(E) = \pi_2(E)$ and $\mathfrak{E}_1^\beta=\mathfrak{E}_2^\beta$ for all $\beta\in A$, and we simply write $\pi(E)$ and $\mathfrak{E}^\beta$. 
	
Notice that throughout the manuscript, we use the identification $\R^m\times\R^m\cong \R^{2m}$ without explicit mention.

Let $C_0(\R^m)$ be the closure with respect to the maximum norm of the space of smooth, real-valued functions on $\R^n$ with compact support.
 By the Riesz representation theorem (see e.g.~\cite[Theorem 1.54]{AFP00}), the dual space of $C_0(\R^m)$ can be identified  via the duality pairing $\langle \mu, \varphi \rangle = \int_{\R^m} \varphi(\xi)\, d\mu(\xi)$ with the space $\mathcal M(\R^m)$ of finite signed Radon measures on $\R^m$. 

	For the class of probability measures defined on the Borel sets of $\R^m$, we write $\mathcal Pr(\R^m)$. The barycenter of $\mu\in \mathcal Pr(\R^m)$ is defined by 
	\begin{align}\label{barycenter}
	[\mu] := \langle \mu, {\rm id}\rangle = \int_{\R^m} \xi\, d\mu(\xi),
	\end{align} 
	and ${\rm supp \,\mu}$ stands for the support of $\mu$. 
	If $f: \R^m\to \R$ and $\mu$ is a probability measure, or more generally, a positive measure, on the Borel sets of $\R^m$, the $\mu$\text{-}essential supremum of $f$ over the set $A\subset \R^m$ is defined as 
	\begin{align*}
	\mu\text{-}\supess_{\xi\in A} f(\xi) := \inf_{\mathcal N \subset A, \mu(\mathcal N)=0}\sup_{\xi\in A \setminus\mathcal N} f(\xi).
	\end{align*}
We use the notation $\nu\otimes \mu$ to denote the product measure of two measures $\nu$ and $\mu$. 
 By $U$ we denote a generic measurable (Lebesgue or Borel) subset of $\mathbb R^m$. 
The Lebesgue measure of a Lebesgue measurable set $U\subset\mathbb R^n$ is denoted by ${\mathcal L}^n(U)$. We skip the Lebesgue measure symbol $\mathcal L^n$ whenever it is clear from the context, for example, we often write simply~`a.e.~in $U$' instead of `$\mathcal L^n$-a.e.~in $U$'. 
	
Unless mentioned otherwise, $\Omega$ is always a non-empty, open and bounded subset of $\mathbb R^n$. We use standard notation for $L^p$-spaces with $p\in [1,\infty]$; in particular, for a sequence of functions $(u_j)_j\subset L^p(\Omega;\R^m)$ and $u\in L^p(\Omega;\R^m)$, we write $u_j\weakly u$ in $L^p(\Omega;\R^m)$ with $p\in [1, \infty)$ and $u_j \weaklystar u$ in $L^\infty(\Omega;\R^m)$ to express weak and weak$^\ast$ convergence of $(u_j)_j$ to $u$, respectively.  
	In the following, we often deal with functions $u \in L^p(\Omega;\R^m)$ and their composition with Borel measurable functions $f:\R^m\to \R$.
The $\mathcal L^n$-essential supremum of $f(u)$, whenever $f$ is non-negative, 
corresponds to the $L^\infty$-norm of $f(u)$.
Depending on the context, we write either $\mathcal{L}^n\text{-}\supess_{x\in \Omega} f(u(x))$,  $\|f(u)\|_{L^\infty(\Omega)}$, or simply, $\supess_{x\in \Omega} f(u(x))$. 		
	
	\subsection{Young measures}\label{subsec:YM}
	Young measures are an important technical tool in nonlinear analysis, as they encode refined information on the oscillation behavior of weakly converging sequences. To make this article self-contained, we briefly recall some basics from this theory, focusing on what will be used in the sequel. For a more detailed introduction to the topic, we refer to the broad literature, e.g.~\cite[Chapter~8]{FoL07}, \cite{Ped97a}, \cite[Section~4]{Rin18}.

	Let $U\subset \mathbb R^n$ be a Lebesgue measurable set with finite measure. 
By definition, a Young measure $\nu=\{\nu_x\}_{x\in U}$ is an element of the space $L^\infty_w(U;\mathcal M(\R^m))$ of essentially bounded, weakly$^\ast$ measurable maps $U\to \mathcal M(\R^m)$, which is isometrically isomorphic to the dual of $L^1(U;C_0(\R^m))$, such that $\nu_x:=\nu(x) \in \mathcal Pr(\R^m)$ for $\mathcal L^n$-a.e.~$x \in U$.  
One calls $\nu$ 
 homogeneous if there is a measure $\nu_0 \in  \mathcal Pr(\mathbb R^m)$ such that $\nu_x = \nu_0$ for $\mathcal L^n$- a.e.~$x \in U$.
	
	A sequence $(z_j)_j$ of measurable functions $z_j: U\to \R^m$ is said to generate a Young measure $\nu \in L^\infty_w(U;\mathcal Pr(\R^m))$ 
	if for every $h \in L^1(U)$ and $\varphi \in C_0(\R^m)$,
	$$
	\lim_{j\to\infty}\int_U h(x)\varphi(z_j(x))\,dx = \int_U h(x)\int_{\R^m}\varphi(\xi)d \nu_x(\xi)\,dx = \int_U h(x) \langle \nu_x, \varphi \rangle\,dx, 
	$$  
	or $\varphi(z_{j})\overset{\ast}{\rightharpoonup} \langle \nu_x,\varphi\rangle$ for all $\varphi \in C_0(\mathbb R^m)$; 
in formulas,
	\begin{align*}
	z_j\stackrel{YM}{\longrightarrow} \nu \qquad \text{as $j\to \infty$.}
	\end{align*}

	The following result is often referred to as the fundamental theorem for Young measures, see e.g.~\cite{Bal89}, \cite[Theorems~8.2 and~8.6]{FoL07}, \cite[Theorem~4.1, Proposition~4.6]{Rin18}.

	\begin{Theorem} 
		\label{FTYM}
	Let $(z_j)_j \subset L^p(U;\R^m)$ with $1\leq p\leq \infty$ be a uniformly bounded sequence. Then  there exists a subsequence of 
	$(z_{j})_j$ (not relabeled) and a Young measure $\nu\in L_w^\infty(U;\mathcal M(\R^m))$ such that $z_j\stackrel{YM}{\longrightarrow} \nu$.  Moreover,
		\begin{itemize}
		\item[$(i)$] for any continuous integrand $f : \mathbb R^m \to \R$ with the property that $\bigl(f(z_j)\bigr)_j\subset L^1(U)$ is equiintegrable, it holds that 
						$$
			f(z_j)\rightharpoonup \int_{\mathbb R^m}f(\xi) \,d\nu(\xi) = \langle \nu, f\rangle \quad \hbox{ in }L^1(U); 
			$$
			\item[$(ii)$] for any lower semicontinuous $f : \mathbb R^m \to \R_\infty$ bounded from below, 
			$$\liminf_{j\to \infty}\int_U f(z_{j} (x)) \,dx \geq\int_{U}\int_{\mathbb R^m}f(\xi) d\nu_x(\xi)\,dx = \int_U \langle \nu_x, f\rangle \, dx;$$
\item[$(iii)$]				if $K \subset \mathbb R^m$ is a compact subset, then ${\rm supp\,} \nu_x \in K$ for $\mathcal L^n$-a.e.~$x \in U$ if and only if  ${\rm dist} (z_j, K)\to 0$ in measure.
		\end{itemize} 
		
	\end{Theorem}

In particular,  
  if $(z_j)_j \subset L^p(U;\R^m)$ generates a Young measure $\nu$ 
  and converges weakly($^\ast$) in $L^p(U;\R^m)$  to a limit function $u$, then $[\nu_x] = \langle \nu_x, {\rm id}\rangle = u(x)$ for $\mathcal L^n$-a.e.~$x\in U$.

With the aim of analyzing nonlocal problems, we associate with any function $u\in L^1(\Omega;\R^m)$ the vector field 
\begin{align}\label{vu}
v_u(x,y):=(u(x), u(y)) \quad \text{for $(x,y)\in \Omega\times \Omega$.}
\end{align}
The following lemma, which was established by Pedregal in~\cite[Proposition 2.3]{Ped97}, gives a characterization of Young measures generated by sequences of such nonlocal vector fields. 
\begin{Lemma}
	\label{asprop2.3Pedregal}
	Let $(u_j)_j \subset L^p(\Omega;\R^m)$ with $1 \leq p\leq \infty$ generate a Young measure $\nu=\{\nu_x\}_{x\in \Omega}$, and let $\Lambda = \{\Lambda_{(x,y)}\}_{(x, y)\in \Omega\times \Omega}$ be a family of probability measures on $\mathbb R^m \times \mathbb R^m$. 

Then $\Lambda$ is the Young measure generated by the sequence $(v_{u_j})_j\subset L^p(\Omega\times \Omega;\R^m\times \R^m)$ defined according to~\eqref{vu} 
if and only if
\begin{equation*} 
\Lambda_{(x,y)} = \nu_x \otimes \nu_y \qquad \text{for 
a.e.~$(x, y)\in \Omega\times \Omega$}\end{equation*}
and
\begin{align*}
\begin{cases}
\displaystyle \int_{\Omega }\int_{\mathbb R^m} |\xi|^p \,d \nu_x(\xi)\,dx< \infty, & \hbox { if }p<\infty,\\[0.2cm]
{\rm supp}\,\nu_x \subset  K  \hbox{ for $\mathcal L^n$-a.e.~}x \in \Omega \hbox{ with a fixed compact set $K\subset \R^m$}, & \hbox{ if } p = \infty.
\end{cases}
\end{align*}
\end{Lemma}

\subsection{Supremal functionals and level convexity}\label{subsec:supremal}

Next, we collect some basic properties and useful results from the theory of supremal functionals, i.e.,~functionals $F:L^\infty(\Omega;\mathbb R^m) \to \R_\infty$ given by
\begin{equation}
\label{(1)}
F(u):=\supess_{x\in \Omega} f(u(x)), 
\end{equation}
where $f:\mathbb R^m \to \R_\infty$ is a Borel measurable function bounded from below. For the relevance of $L^\infty$-functionals in optimal control and optimal transport problems, see \cite{Bar99, BBJ} and the references therein; applications in the context of materials science can be found e.g. in \cite{BN, GNP, KL}.

Barron \& Jensen in \cite{BaJ95} and Barron \& Liu in \cite{BaL97} were the first to study necessary and sufficient conditions of supremal functionals as $F$ in~\eqref{(1)}. Assuming that $\Omega\subset \R$ is an interval, they proved that $F$ is sequentially $L^\infty$-weakly$^\ast$ lower semicontinuous if and only if the supremand $f$ is level convex and lower semicontinuous. The same statement holds for general $\Omega\subset \R^n$; see~\cite[Theorem 4.1]{ABP02}, as well as \cite{BJW01} and \cite{Pri09}. 
	\begin{Definition}
		\label{levelconvexity} 
	A function $f: \mathbb R^m \to \R_\infty$ is called level convex
		if all level sets of $f$, that is,~$L_c(f) = \{\xi\in \R^m: f(\xi)\leq c\}$ with $c\in \R$, are convex sets.
	\end{Definition}
Note that level convexity is known in the literature on operational research and convex analysis as quasiconvexity, see e.g.~\cite{Mangasarian94}. To avoid ambiguity with the notion introduced by Morrey~\cite{Mor66} in the context of integral functionals, we have chosen here to use the same terminology as in~\cite{ABP02}.

The following lemma provides different characterizations of level convexity, in particular, in terms of a supremal Jensen type inequality. It can be found e.g.~in \cite[Theorem 30]{Bar99} (under additional lower semicontinuity hypotheses) and partially in \cite[Lemma 2.4]{BJL96} and \cite[Theorem 1.2]{BJW01}; see also \cite[Definition 2.1 and Theorem 2.4]{RiZ14} for a statement in wider generality. 
\begin{Lemma}
\label{Theorem24RZ} 
Let 
$f:\mathbb  R^m \to \R_\infty$ be a Borel measurable function. 
Then the following statements are equivalent: 
\begin{itemize}
\item[$(i)$] $f$ is level convex;
\item [$(ii)$] for every $\xi, \zeta \in  \mathbb R^m$ and $t \in [0, 1]$ it holds that 
		$$f(t\xi + (1-t)\zeta) \leq \max\{f(\xi), f(\zeta)\};$$
	
\item[$(iii)$] for any open set $\Omega\subset \R^n$ with $\mathcal L^n(\Omega)<\infty$ and every $\varphi\in L^1(\Omega;\R^m)$ one has that 
$$
f\left(\frac{1}{\mathcal L^n(\Omega)}\int_{\Omega}\varphi \,dx \right) \leq  \supess_{x \in \Omega}f(\varphi(x));
$$	
\item[$(iv)$]  for every $\mu \in \mathcal Pr(\R^m)$,
		$$	f([\mu])  \leq \mu\text{-}\supess_{\xi \in \R^m} f(\xi).$$ 
\end{itemize}
\end{Lemma}

The following auxiliary result is a slight modification of~\cite[Theorem~34]{Bar99} and is based on $L^p$-approximation in combination with the lower semicontinuity type result for Young measure in~Theorem~\ref{FTYM}. 

\begin{Lemma}
	\label{thm34Barronmon} 
Let 
$f:\R^m\to \R_\infty$ a lower semicontinuous function bounded from below. 
	Further, let $(u_j)_j$ be a uniformly bounded sequence of functions in $L^\infty(\Omega;\mathbb R^m)$ generating a Young measure $\nu=\{\nu_x\}_{x\in \Omega}$. 	
	Then,
	\begin{equation*}
	\liminf_{j \to \infty}\supess_{x\in \Omega} f(u_j)  \geq \supess_{x\in \Omega}\bar f,
	\end{equation*}
	where
	$\bar f(x) := \nu_x\text{-}\supess_{\xi\in \R^m} f(\xi) 
	$
	for $x\in\Omega$.
\end{Lemma}
\begin{proof} 
We give the details here for the reader's convenience, referring to~\cite{Bar99} for the original proof.  
Up to a translation argument, there is no loss of generality in assuming that $f$ is non-negative. 

Let $\varepsilon>0$ be fixed, and choose a set $S\subset \Omega$ with positive Lebesgue measure such that $\bar{f}(x)\geq \|\bar{f}\|_{L^\infty(\Omega)}-2\varepsilon$ for all $x\in S$.
Next, we show that there exists a measurable subset $S'\subset S$ with $\mathcal L^n(S')>0$ such that
\begin{align}\label{est1}
\Bigl(\int_{\R^m} |f(\xi)|^p \,d\nu_x(\xi)\Bigr)^{\frac{1}{p}} \geq \|\bar f\|_{L^\infty(\Omega)}-\varepsilon
\end{align}
for all $x\in S'$ and $p>1$ sufficiently large. 
Indeed, with 
$$S_j := \Bigl\{x \in S: 
	\bigl(\textstyle \int_{\Omega}f(\xi)^p d\nu_x(\xi)\bigr)^{\frac{1}{p}}
	\geq  \| \bar f\|_{L^\infty(\Omega)} - \varepsilon \text{ for all $p \geq j$}\Bigr\}$$
for $j \in \mathbb N$, one has that $S= \bigcup_{j=1}^\infty S_j$.
Since $\mathcal L^N(S) > 0$, there must be at least one $j'$ for which $\mathcal L^N(S_{j'}) > 0$, and setting $S' := S_{j'}$
shows~\eqref{est1}. 

We take the inequality in~\eqref{est1} to the $p$th power and integrate over $S'$. Along with Theorem \ref{FTYM}\,(ii),  
it follows that
\begin{align*}
\mathcal L^n(S') (\|\bar{f}\|_{L^\infty(\Omega)}-\varepsilon)^p & \leq \int_{S'} \int_{\R^m} |f(\xi)|^p\, d\nu_x(\xi) \, dx 
\\ &\leq \liminf_{j\to \infty} \int_{\Omega} |f(u_j)|^p \, dx \leq \liminf_{j\to \infty} \|f(u_j)\|_{L^\infty(\Omega)}^p \mathcal L^n (\Omega).
\end{align*}
Hence, 
\begin{align*}
\liminf_{j\to \infty} \|f(u_j)\|_{L^\infty(\Omega)} \geq  \Bigl(\frac{\mathcal L^n(S')}{\mathcal L^n (\Omega)}\Bigr)^{\frac{1}{p}}\bigl(\|\bar{f}\|_{L^\infty(\Omega)}-\varepsilon\bigr)
\end{align*}
for $p>1$ sufficiently large. Letting $p\to \infty$ and recalling that $\varepsilon>0$ is arbitrary concludes the proof.
\end{proof}

\subsection{Double-integral functionals and separate convexity}\label{subsec:nonlocal}

This subsection presents some preliminaries on nonlocal integral functionals, see also~\cite{Ped16} for a recent overview article. 
For $p >1$, consider a double-integral functional $I: L^p(\Omega;\R^m)\to \R$,
\begin{align}\label{doubleintegral_Pedregal}
I(u):=  \int_{\Omega}\int_{\Omega} W(u(x), u(y)) \,dx\, dy,
\end{align}
where $W:\R^m\times \R^m\to \R$ is a continuous function that is bounded from below and has standard $p$-growth.  

In 1997, Pedregal~\cite{Ped97} gave the first necessary and sufficient condition for $L^p$-weak lower semicontinuity of $I$ in the scalar case $m=1$.  This condition was quite implicit, but could be shown to be equivalent to the separate convexity of the integrand $W$ a decade later by Bevan \& Pedregal \cite{BeP06}. Also in the vectorial case, $W$ being separately convex is the characterizing property to ensure weak lower semicontinuity of $I$, as Mu\~{n}oz~proved in~\cite{Mun09}; the latter is formulated in the gradient setting, using $W^{1,p}$-weak convergence of scalar valued functions, but the statement and the ideas of the proof carry over to functionals of the form~\eqref{doubleintegral_Pedregal}, cf.~\cite{Ped16}.
Results about inhomogeneous double-integral functionals, meaning~with integrands $W$ depending also explicitly on $x, y\in \Omega$, can be found e.g.~in \cite{BMC18, Mun09, Ped16}.

\begin{Definition}
	\label{separateconvexityfunctions}
We call a function $W:\mathbb R^m \times \mathbb R^m \to \R_\infty$ separately convex (with vectorial components) if for every $\xi \in \mathbb R^m$, the functions $W(\cdot, \xi)$ and $W(\xi,\cdot)$ are convex.
\end{Definition}

Besides our terminology, which is inspired by~\cite{Dac08}, other names for separate convexity are common in the literature, such as  orthogonal convexity, directional convexity
or bi-convexity; see~\cite{AuH86}, for the first detailed treatment of the subject. 

As discussed recently in~\cite{BMC18}, there are different `nonlocal' definitions of convexity related to the weak lower semicontinuity of $I$, which coincide under suitable assumptions. In Section~\ref{sec:indicator}, we extend the discussion of these notions to the context of unbounded functionals.

It was observed in \cite[p.~1383]{Ped97} that for $W:\mathbb R\times \mathbb R \to \R$ continuous and bounded from below, separate convexity of $W$ can equivalently be characterized by a separate Jensen's inequality. 
In view of~\cite[Theorem 4.1.4]{CDA02}, this statement can easily be generalized to extended-valued, lower semicontinuous functions defined on $\R^m\times \R^m$ as follows. 

\begin{Lemma}
	\label{lem:sepJensen}
Let $W:\mathbb R^m \times \mathbb R^m \to \R_\infty$ be lower semicontinuous and bounded from below, then
$W$ is separately convex if and only if 
\begin{align}
\label{sepJensen}
\int_{\mathbb R^m}\int_{\mathbb R^m} W(\xi,\zeta)  \,d \nu(\xi) \,d  \mu(\zeta) \geq 
W([\nu], [\mu])
\end{align}
for any  $\mu, \nu \in \mathcal Pr(\R^m)$. 
\end{Lemma} 
\begin{proof}
	[Proof] Assuming first that $W$ is separately convex, 
to obtain~\eqref{sepJensen}, it suffices now to apply Jensen's inequality in the version of~\cite[Theorem 4.1.4]{CDA02} twice; first with the integrand $W(\cdot, \xi)$ for $\mu$-a.e.~$\xi\in \R^m$, and then with $W([\nu], \cdot)$.
 
The fact that \eqref{sepJensen} yields separate convexity of $W$ follows after choosing $\mu$ and $\nu$ to be convex combinations of Dirac measures. 
\end{proof}
The question of relaxation of functionals $I$ as in~\eqref{doubleintegral_Pedregal} for which the density $W$ fails to be separately convex is still mostly open. 
It may seem counter-intuitive, but there are examples~\cite{BMC18, BeP06, Ped16} indicating that separate convexification of $W$ 
does in general not give rise to the right candidate for the weakly lower semicontinuous envelope of $I$. Even more remarkably, as recently proven in \cite{KrZ19, MCT20}, relaxation in the weak $L^p$-topology of double-integrals functionals cannot always be expected to be structure-preserving. In the context of Young measures, we refer to~\cite{BMC18} for a relaxation result with respect to the narrow convergence.

\section{Separate level convexity}\label{3}
In this section, we introduce the notion of separate level convexity, 
and show that it provides a sufficient condition for the $L^\infty$-weak$^\ast$ lower semicontinuity of nonlocal supremal functionals as in~\eqref{ourfunct}. 

Before doing so, let us specify what we mean by separate convexity with vectorial components (in the sequel, just referred to as {separate convexity}) of subsets of $\R^m\times \R^m$. 

For $m=1$, this definition reduces to classical separate convexity in the sense of~\cite[Proposition 7.5 and Definition 7.13]{Dac08}. 

\begin{Definition}[Separate convexity (with vectorial components) of sets]
	\label{separateconvexitysets}
A set $E\subset \R^m\times \R^m$ is called separately convex, if for every $t\in (0,1)$ and every $(\xi_1, \zeta_1), (\xi_2, \zeta_2)\in E$ such that $\xi_1=\xi_2$ or $\zeta_1=\zeta_2$ it holds that
	\begin{align*}
	t(\xi_1, \zeta_1) + (1-t)(\xi_2, \zeta_2) \in E.
	\end{align*}

The separately convex hull of $E$, denoted by $E^{\rm sc}$, is defined as the smallest separately convex set in $\mathbb R^m\times \R^m$ containing $E$.
\end{Definition}

The separately convex hull of $E\subset\R^m\times \R^m$ can be characterized by 
\begin{align}\label{separatelyconvex_hull}
E^{\rm sc} = \bigcup_{i\in \mathbb{N}}E^{\rm sc}_i
\end{align}
with $E_0^{\rm sc} = E$ and for $i\in \mathbb N$,
\begin{align}\label{construction_sc}
E_i^{\rm sc} &= \{(\xi, \zeta)\in \R^{m}\times \R^m: (\xi, \zeta)=t (\xi_1, \zeta_1)+(1-t) (\xi_2, \zeta_2), t\in [0,1], \\ & \qquad\qquad \qquad \hspace{0.5cm}
(\xi_1, \zeta_1), (\xi_2, \zeta_2)\in E_{i-1}^{\rm sc}, \xi_1=\xi_2 \text{ or } \zeta_1=\zeta_2 \},\nonumber
\end{align}
cf.~\cite[Theorem~7.17]{Dac08}. 

\begin{Remark}\label{rem:Esc_compact}
It is clear by the construction in~\eqref{separatelyconvex_hull} and~\eqref{construction_sc} that if $E$ is open, then so is $E^{\rm sc}$. While compactness of $E$ is preserved under separate convexifications in the two-dimensional setting (i.e.~if $m=1$) as stated in~\cite[Proposition~2.3]{Kol03}, this is in general not true for $m>1$ \cite[Remark~7.18 (ii)]{Dac08}; more details on the latter are given in~\cite{AuH86, Kol03}.  
\end{Remark}

\begin{Definition}[Separate level convexity (with vectorial components) of functions]
	\label{seplevconv} 
	We call a function $W: \mathbb R^m\times \mathbb R^m \to \R_\infty$ {separately level convex} if 
	all level sets of $W$, i.e.~the sets $L_c(W) = \{(\xi,\eta)\in\mathbb R^m\times \mathbb R^m: W(\xi,\eta)\leq c\}$ with $c\in\R$, are separately convex. 
	
	Furthermore, $W^{\rm slc}$ stands for the separately level convex envelope of $W$, that is, the largest separately level convex function below $W$.	
\end{Definition}

\begin{Remark}
	\label{prop:levelsetlsc}
	a) An equivalent way of expressing separate level convexity of $W:\mathbb R^m \times \mathbb R^m\to \R_\infty$ is that for every $\xi, \zeta\in \R^m$, the functions $W(\xi, \cdot), W(\cdot, \zeta):\R^m\to \R_\infty$ are level convex.

b) In view of the above definitions, we observe that 
\begin{align}\label{inclusion_levelconvexenvelope}
L_c(W^{\rm slc}) \supset L_c(W)^{\rm sc} \hbox{ for any }c \in \mathbb R.
\end{align} 
In general, equality in~\eqref{inclusion_levelconvexenvelope} is not true as the example 
\begin{align*}
\mathbb R\times \mathbb R \ni(\xi, \zeta)\mapsto W(\xi,\zeta)=\left\{\begin{array}{ll}|(\xi, \zeta)| &\hbox{ if } (\xi,\zeta)\neq (0,0), \\
1 &\hbox{ if }(\xi,\zeta)=(0,0),
\end{array}\right.
\end{align*}
shows. Here, $L_0(W^{\rm slc}) = \{0\}$, whereas $L_0(W)^{\rm sc} = \emptyset$. Under additional assumptions, equality in~\eqref{inclusion_levelconvexenvelope} is nevertheless true, cf.~\eqref{slcenvelope=}. 
\end{Remark}

The following lemma collects a number of different representations of separate level convexity. \begin{Lemma}\label{lem:separate_level_convexity}
Let $W:\R^m\times \R^m\to \R_\infty$ be Borel measurable. Then the following statements are equivalent: 
\begin{itemize}
\item[$(i)$] $W$ is separately level convex;
\item [$(ii)$] for every $\xi_1, \xi_2, \zeta_1, \zeta_2 \in  \mathbb R^m$ and $t, s \in [0, 1]$ one has that
	\begin{align*}
	W(t\xi_1 + (1-t)\xi_2, s\zeta_1+(1-s) \zeta_2) &\leq \max_{i,j\in \{1,2\}} W(\xi_i,\zeta_j);
	\end{align*}
	\item[$(iii)$] for any open $\Omega\subset \R^n$ with $\mathcal L^n(\Omega)<\infty$ and all $\varphi, \psi\in L^1(\Omega;\mathbb R^m)$, 
	$$ W\Bigl(\frac{1}{\mathcal L^n(\Omega)}\int_{\Omega}\varphi \,dx, \frac{1}{\mathcal L^n(\Omega)}\int_{\Omega}\psi \,dy \Bigr) \leq  \supess_{(x,y) \in \Omega\times \Omega}W(\varphi(x),\psi(y));$$ 
	 \item[$(iv)$]  for every $\nu, \mu\in \mathcal Pr(\R^m)$ it holds that
\begin{align*}
W([\nu], [\mu])  & \leq  (\nu\otimes\mu)\text{-}\supess_{(\xi, \zeta)\in \R^m\times \R^m} W(\xi, \zeta) \\ &= \nu\text{-}\supess_{\xi \in \R^m} \bigl(\mu\text{-}\supess_{\zeta \in \R^m} W(\xi, \zeta)\bigr) = \mu\text{-}\supess_{\zeta \in \R^m} \bigl(\nu\text{-}\supess_{\xi \in \R^m} W(\xi, \zeta)\bigr).
\end{align*}
\end{itemize}
\end{Lemma}

\begin{proof}
These equivalences follow as an immediate corollary of~Lemma~\ref{Theorem24RZ}. Indeed, we apply the characterizations therein twice in each of the two variables of $W$, fixing the other. 
\end{proof}

The sufficiency of separate level convexity of $W$ for ensuring $L^\infty$-weak$^\ast$ lower semicontinuity of $J$ in \eqref{ourfunct} follows in light of the coercivity assumption of $W$ and
 Remark~\ref{rem1.2}\;a) from the next proposition. The proof relies on combining elements from both theories of supremal and double-integral functionals, cf.~Sections~\ref{subsec:YM} and~\ref{subsec:supremal}, respectively. 

\begin{Proposition}
	\label{suffseplevconv}
	Let $J$ be as in \eqref{ourfunct} with $W:\mathbb R^m \times \mathbb R^m \to \mathbb R$ lower semicontinuous and coercive, i.e., $W(\xi, \zeta)\to \infty$ as $|(\xi,\zeta)|\to \infty$. If $W$ is separately level convex,
	then $J$ is 
	$L^\infty$-weakly$^\ast$ lower semicontinuous, i.e.,~for all $(u_j)_j\subset L^\infty(\Omega;\R^m)$ and $u\in L^\infty(\Omega;\R^m)$ such that $u_j\weaklystar u$ in $L^\infty(\Omega;\R^m)$ it holds that 
	\begin{equation*}
	\liminf_{j\to \infty} \supess_{(x,y)\in \Omega\times \Omega}W(u_j(x),u_j(y)) \geq \supess_{(x, y )\in \Omega \times \Omega}W(u(x),u(y)).
	\end{equation*}
\end{Proposition}

\begin{proof}[Proof]
	Let $(u_j)_j\subset L^\infty(\Omega;\R^m)$ be such that $u_j\weaklystar u$ in $L^\infty(\Omega;\R^m)$ and let $\nu=\{\nu_x\}_{x\in \Omega}$ be the Young measure generated by $(u_j)_j$ (possibly after passing to a non-relabeled subsequence). In particular, 
	\begin{equation}
	\label{eq1}[{\nu_x}]=\langle \nu_x,{\rm id} \rangle = u(x)\qquad \text{for a.e.~$x\in \Omega$.}
	\end{equation} 
	Let $(v_{u_j})_j\subset L^\infty(\Omega\times \Omega;\R^m\times \R^m)$ be the sequence of nonlocal vector fields associated with $(u_j)_j$, cf.~\eqref{vu}, and $\Lambda=\{\Lambda\}_{(x,y)\in \Omega\times \Omega}=\nu_x\otimes \nu_y$ for $x,y\in \Omega$ the generated Young measure according to  Lemma~\ref{asprop2.3Pedregal}. 
	Then, Lemma \ref{thm34Barronmon} implies that
	\begin{equation}
	\label{one}
	\liminf_{j\to \infty} J(u_j) = \liminf_{j\to \infty} \supess_{(x,y)\in \Omega\times \Omega}W(u_j(x),u_j(y))\geq \supess_{(x,y)\in \Omega\times \Omega} \overline W(x,y),
	\end{equation}
	where $\overline W(x,y) := \Lambda_{(x,y)}\text{-}\supess_{(\xi,\zeta)\in \mathbb R^m \times \mathbb R^m} W(\xi, \zeta)$. By Lemma \ref{asprop2.3Pedregal}, 
	\begin{equation}
	\nonumber
	\overline W(x,y)=  \nu_x\otimes \nu_y\text{-}\supess_{(\xi,\zeta)\in \mathbb R^m \times \mathbb R^m} W(\xi, \zeta)=
	\nu_x\text{-}\supess_{\xi\in \mathbb R^m}\bigl(\nu_y\text{-} \supess_{\zeta \in \mathbb R^m} W(\xi, \zeta) \bigr)
	\end{equation}
	for a.e.~$(x, y)\in \Omega\times \Omega$, and since $W$ is separately convex, Lemma \ref{lem:separate_level_convexity}\,$(iv)$ along with~\eqref{eq1} guarantees that
	\begin{equation}
	\label{two}
	\overline W(x,y)
	\geq W([\nu_x],[\nu_y]) = W(u(x), u(y)).
	\end{equation}
	
Joining \eqref{two} and \eqref{one} concludes the proof.
\end{proof}

As we show later in Section~\ref{subsec:lsc_relaxation}, separate level convexity of $W$ is not necessary for $J$ being sequentially $L^\infty$-weakly$^\ast$ lower semicontinuous, cf.~Corollary~\ref{cor}. 

\section{Diagonalization, symmetrization and separately convex hulls}\label{sec:diagonal}

For $E \subset \mathbb R^m \times \mathbb R^m$, let
\begin{align*}
E^{\rm diag} :=\{(\alpha, \beta)\in E: (\alpha,\alpha), (\beta,\beta)\in E\}
\end{align*} 
and
\begin{align*}
E^{\rm sym} :=\{(\alpha, \beta)\in E: (\beta, \alpha)\in E\}=E\cap E^T
\end{align*} 
be the diagonalization and symmetrization of $E$. Accordingly, we call $E$ symmetric, if $E =E^{\rm sym}$, and diagonal if $E=E^{\rm diag}$. 
By combining these two operations, we introduce
\begin{align}\label{hatK}
\widehat{E} & := E^{\rm sym}\cap E^{\rm diag} \\ & \nonumber = (E^{\rm diag})^{\rm sym} =  (E^{\rm sym})^{\rm diag}  = \{(\alpha, \beta)\in E: (\alpha, \alpha), (\beta, \alpha), (\beta, \beta)\in E\}.
\end{align}

As an immediate consequence of these definitions, one observes that if $E$ is closed (compact), then $E^{\rm sym}$ and $E^{\rm diag}$, and consequently, also $\widehat{E}$, are closed (compact).

This section is devoted to the study of characterizing properties of diagonal and symmetric sets. 
For illustration, we start with 
a few simple examples in the scalar case $m=1$.

\begin{Example}\label{exampleK} 
Consider the four compact subsets of $\R\times \R$,
\begin{align*}
\begin{array}{l}
\hspace{0.6cm}K_1= [-2,2]\times [-1,1], \quad K_2=\{(\xi,\zeta)\in \R\times \R: \xi^2+\zeta^2\leq 2\},\quad \\[0.2cm]
K_3  =\{(\xi, \zeta)\in \R\times \R: |\xi|+|\zeta|\leq 2\},\quad \text{and}\quad K_4 = [-1, 1]\times [-1,1]. 
\end{array} 
\end{align*}
Then, $\widehat{K_1} =\widehat{K_2}=\widehat{K_3}=\widehat{K_4}=K_4$. For the points sets 
\begin{align}\label{K4K5}
K_5=\{(1,0), (0,1), (-1, 0), (0, -1)\}\quad\text{and}\quad  K_6=\{-1,1\}\times \{-1,1\},
\end{align}
one obtains that $\widehat{K_5} =\emptyset$ and $\widehat{K_6} = K_6$, respectively. 
\end{Example}

Notice the following equivalent way of expressing $\widehat E$ in \eqref{hatK},
\begin{align}\label{tildeK_construction}
\widehat E = E^{\rm sym} \setminus B_E \qquad
 \text{with } \ B_E:= \bigcup_{(\xi,\xi)\notin E} (\R^m\times \{\xi\}) \cup (\{\xi\}\times \R^m).
\end{align}

Based on the concept of maximal Cartesian subsets and motivated by the observation that $\widehat E=\bigcup_{(\xi, \zeta)\in \widehat E} \{\xi, \zeta\}\times \{\zeta, \xi\}\subset \bigcup_{(\xi, \zeta)\in E} \{\xi, \zeta\}\times \{\zeta, \xi\}$, we will derive yet another representation of $\widehat E$ in Lemma~\ref{lem:Khat_alternative}. 
\begin{Definition}\label{def:Cartesian_subset}
Let  $E\subset \R^m\times \R^m$. We call a set $P\subset E$ a maximal Cartesian subset of $E$ if $P=A\times A$ with $A\subset \R^m$ and if for any $B\subset \R^m$ with $A\subset B$ and $B\times B\subset E$ it holds that $B=A$. We denote the set of all maximal Cartesian subsets of $E$ by $\mathcal P_E$.
\end{Definition}

\begin{Lemma}\label{lem:Khat_alternative}
Let $E\subset \R^m\times \R^m$. Then, 
\begin{align*}
\widehat E = \bigcup_{P\in \mathcal P_E} P.
\end{align*}
\end{Lemma}
\begin{proof}
The proof follows simply from exploiting the definitions of $\mathcal P_E$ and $\widehat E$. Here are some more details for the readers' convenience. If $(\xi, \zeta)\in P$ for some $P\in \mathcal P_E$,  then $\{\xi, \zeta\}\times \{\xi, \zeta\}\subset P \subset E$. Hence, $(\xi, \zeta)$, $(\xi, \xi)$, $(\zeta, \xi), (\zeta, \zeta)\in E$, which shows that $(\xi, \zeta)\in \widehat E$. 

On the other hand, we know for $(\xi, \zeta)\in \widehat E$ that $\{\xi, \zeta\}\times \{\xi, \zeta\}\subset \widehat E\subset E$, and hence $B\times B\subset E$ with $B=\{\xi, \zeta\}$. Due to the Cartesian structure of $B\times B$, there is a maximal Cartesian subset of $E$ containing $B\times B$, which proves the statement. 
\end{proof}

\begin{Remark}
It is immediate to see that $\mathcal P_E=\mathcal P_{\widehat E}$.
\end{Remark}

Recalling Definition \ref{separateconvexitysets}, we prove that diagonalization and symmetrization preserves separate convexity if $m=1$. For $m>1$, however, this is in general not true, see Remark~\ref{rem:hat1}\,b).  
\begin{Lemma}
	\label{Khatsepconv}
	If $E\subset \R\times \R$ is separately convex, then $\widehat E$ is also separately convex. 
\end{Lemma}
\begin{proof}[Proof] 
 Let $(\xi_1, \zeta), (\xi_2,\zeta) \in \widehat E$. By Lemma~\ref{lem:Khat_alternative} we know that there are $P_1, P_2\in \mathcal P_E$ 
such that $(\xi_1, \zeta)\in P_1=A_1\times A_1$ and $(\xi_2, \zeta)\in P_2 = A_2\times A_2$ with $A_1, A_2\subset \R$.  Since $E$ is separately convex, $A_1, A_2\subset \R$ are convex, and hence intervals. Observing that $\zeta\in A_1\cap A_2$, the intervals overlap, so that $(A_1\cup A_2)^{\rm co}= A_1\cup A_2$. Consequently, any convex combination $t\xi_1+(1-t)\xi_2$ with $t\in [0,1]$ lies in $A_1\cup A_2$, which implies $(t\xi_1+(1-t)\xi_2, \zeta)\in P_1\cup P_2\subset \widehat E$, cf.~Lemma~\ref{lem:Khat_alternative}. By Definition~\ref{separateconvexitysets}, $\widehat E$ is thus separately convex. %
	\end{proof}

	\begin{Remark}\label{rem:hat1} a) Due to Lemma~\ref{Khatsepconv}, it holds that $\widehat{E}^{\rm sc}\subset \widehat{E^{\rm sc}}$ for any $E\subset \R\times \R$. We point out, however, that the operations of taking the separate convexification and diagonalization of $E\subset \R\times \R$ do in general not commute, that is, $\widehat{E}^{\rm sc}\neq \widehat{E^{\rm sc}}$. In fact, the set $K_5$ in~\eqref{K4K5} satisfies $\widehat{K_5}^{\rm sc}=\emptyset$, while $\widehat{K_5^{\rm sc}} = ([-1,1]\times \{0\}\cup \{0\}\times [-1,1])^{\rm diag} = \{0\}$. 
\\[-0.2cm]

 b) Note that the statement of Lemma~\ref{Khatsepconv} fails in the vectorial case $m>1$, as the following example illustrates. Let $E=(A_1\times A_1) \cup (A_2\times A_2)$ with $A_1, A_2\subset \R^m$ convex such that $A_1\cap A_2\neq \emptyset$ and $(A_1\cup A_2)^{\rm co} \setminus (A_1\cup A_2)\neq \emptyset$. Then,
\begin{align}\label{eq86}
E^{\rm sc}= \color{olive} E^{\rm sc}_m \color{black} =E\cup [(A_1\cap A_2)\times (A_1\cup A_2)^{\rm co}]\cup [(A_1\cup A_2)^{\rm co}\times (A_1\cap A_2)],
\end{align}
and hence, in view of $E=\widehat E$, we find that
 $\widehat{E^{\rm sc}}=E$. Since $E$ is strictly contained in $E^{\rm sc}$, however, $E$ is not separately convex.  
\end{Remark}

The next lemma gives a characterization of the separate convex hull of
 symmetric and diagonal sets in the scalar case $m=1$. 

\begin{Lemma}\label{lem:square}
	Let $E\subset \R\times \R$ be symmetric and diagonal. 
Then
	\begin{align}\label{eq5}
	E^{\rm sc} = \bigcup_{(\alpha, \beta)\in  E
	}Q_{\alpha, \beta},
	\end{align}	
recalling that $Q_{\alpha, \beta}=[\alpha, \beta]\times [\alpha, \beta]$ for $\alpha, \beta\in \R$, where $[\alpha, \beta]\subset \R$ stands for the generalized interval in the sense of~\eqref{generalizedinterval}. 

\color{olive} Moreover, if $E^{\rm sc}$ is compact, then $\mathcal P_{E^{\rm sc}}\subset \{Q_{\alpha, \beta}: (\alpha, \beta)\in E\}$. 
\end{Lemma}

\begin{proof} 
	For any $(\alpha, \beta)\in E=\widehat E$, 
	we have that $\{\alpha, \beta\}\times\{\alpha, \beta\}\subset E$, so that 
	\begin{align*}
	Q_{\alpha,\beta} = \{\alpha, \beta\}^{\rm co} \times \{\alpha, \beta\}^{\rm co}= (\{\alpha, \beta\}\times \{\alpha, \beta\})^{\rm sc}\subset E^{\rm sc}.
	\end{align*} 
	Hence, $\bigcup_{(\alpha, \beta)\in E} Q_{\alpha, \beta} \subset E^{\rm sc}$.

For the reverse implication in~\eqref{eq5}, it suffices to observe that $E_Q:=\bigcup_{(\alpha, \beta)\in E}Q_{\alpha, \beta}\supset E$ is separately convex. Indeed, if $(\xi, \zeta_1), (\xi, \zeta_2)\in E_Q$, then $(\xi, \zeta_1)\in Q_{\alpha_1, \beta_1}$ and $(\xi, \zeta_2)\in Q_{\alpha_2, \beta_2}$ with $(\alpha_1, \beta_1), (\alpha_2, \beta_2)\in E$. 
The union of these two overlapping squares contains the line between the points $(\xi, \min\{\alpha_1, \alpha_2\})$ and $(\xi, \max\{\beta_1, \beta_2\})$, and therefore also $(\xi, t\zeta_1+(1-t)\zeta_2)$ for any $t\in (0,1)$. Since $E_Q$ is symmetric, this is enough to conclude the separate convexity of $E_Q$, 
which finishes the proof \color{olive} of~\eqref{eq5}. 

To see the add-on, consider $A\times A\in \mathcal P_{E^{\rm sc}}$. From the compactness of $E^{\rm sc}$ and the maximality property of $A\times A$, we infer that $A\subset \R$ is convex and compact, and hence, a closed interval, say $A=[\xi, \zeta]$ 
with $\xi, \zeta\in \R$ such that $(\xi, \zeta)\in E^{\rm sc}$. 
According to~\eqref{eq5}, there exists $(\alpha, \beta)\in E$ with $(\xi, \zeta)\in Q_{\alpha, \beta}\subset E^{\rm sc}$. Assuming that $(\xi, \zeta)\neq (\alpha, \beta)$ generates a contradiction with the maximality of $A\times A=Q_{\xi, \zeta}$, hence $(\xi, \zeta) = (\alpha, \beta)\in E$. 
\end{proof} 

\begin{Remark}\label{rem:failure}
a) As a consequence of Lemma~\ref{lem:square}, the properties of a symmetric and diagonal set $E\subset \R\times \R$ carry over to its separate convexification $E^{\rm sc}$. 

b) 
In view of~\eqref{eq5}, a Caratheodory type formula holds for separate convex hulls of sets as in Lemma~\ref{lem:square}. In general, this cannot be expected, see e.g.~\cite[Section~2.2.3]{Dac08}. 
Recalling~\eqref{separatelyconvex_hull} and~\eqref{construction_sc}, we have that 
	\begin{align*}
	E^{\rm sc} = E^{\rm sc}_2.
	\end{align*}
	 Indeed, if $(\xi, \zeta)\in E^{\rm sc}$, then~\eqref{eq5} implies that $(\xi, \zeta)\in Q_{\alpha, \beta}$ for some $(\alpha, \beta)\in E$, 
and there are $t, s\in [0,1]$ such that $\xi=t\alpha +(1-t)\beta$ and $\zeta = s\alpha+(1-s)\beta$. Thus,
$
(\xi, \zeta) = t(\alpha, \zeta) + (1-t)(\beta, \zeta),
$
or equivalently,
\begin{align*}
(\xi, \zeta) = ts(\alpha, \alpha)+ t(1-s)(\alpha,\beta) +(1-t)s(\beta, \alpha) + (1-t)(1-s)(\beta, \beta).
\end{align*} 

c) We emphasize that the representation formula~\eqref{eq5} is in general not true in the vectorial case, that is, for symmetric and diagonal subsets of $\R^m\times \R^m$ with $m>1$.  
To see this, consider the example of Remark~\ref{rem:hat1}\,b), where $E$ is the union of two Cartesian products generated by convex sets $A_1, A_2\subset \R^m$ with $m>1$ whose union is not convex. Then, due to the convexity of $A_1$ and $A_2$ and the fact that $E$ is not separately convex, we conclude that
\begin{align*}
E^{\rm sc} \neq E= \bigcup_{(\alpha, \beta)\in E} Q_{\alpha, \beta}. 
\end{align*} 
After diagonalization (and symmetrization), however, we observe that
\begin{align*}
\widehat{E^{\rm sc}} = \bigcup_{(\alpha, \beta)\in E} Q_{\alpha, \beta}  = E. 
\end{align*}

 d) It remains an open question at this point to find an explicit representation for $E^{\rm sc}$, or $\widehat{E^{\rm sc}}$, with general $E\subset \R^m\times \R^m$ symmetric and diagonal.

In \color{olive} a \color{black} special case 
when at most 
two of the separately convex hulls of the maximal Cartesian subsets of $E$ intersect,
we can derive a formula for $\widehat{E^{\rm sc}}$ based on~\eqref{eq86}. 
Precisely, suppose that $E = \bigcup_{P=A\times A\in \mathcal{P}_E} P$ 
and that there are 
$P_1=A_1\times A_1\in \mathcal P_E$ and $P_2=A_2\times A_2\in \mathcal P_E$ with $A_1, A_2\subset \R^m$ such that  
$P^{\rm sc}\cap Q^{\rm sc}=\emptyset$ for all $P\in \mathcal P_E$ and $Q\in \mathcal P_E\setminus\{P, P_1, P_2\}$, \color{olive} and $(A_1\cup A_2)^{\rm co}\cap A^{\rm co} =\emptyset$ for all sets $A\notin \{A_1, A_2\}$ with $A\times A \in \mathcal{P}_E$.  \color{black}

Along with the observation that $(B\times B)^{\rm sc} = B^{\rm co}\times B^{\rm co}$ for any $B\subset \R^m$, it follows that
\begin{align*}
E^{\rm sc} = \bigl[\bigcup_{P=A\times A\in \mathcal{P}_E}  A^{\rm co}\times A^{\rm co}\bigr] \cup [(A_1^{\rm co}\cap A_2^{\rm co})\times (A_1\cup A_2)^{\rm co}] \cup  [(A_1\cup A_2)^{\rm co}\times (A_1^{\rm co}\cap A_2^{\rm co})]. 
\end{align*}
Hence, 
\begin{align*}
\widehat{E^{\rm sc}} = \bigcup_{P\in \mathcal{P}_E}P^{\rm sc} = \bigcup_{P=A\times A\in \mathcal{P}_E} A^{\rm co}\times A^{\rm co} = \bigcup_{P=A\times A\in \mathcal{P}_E}\bigcup_{(\alpha, \beta) \in A^{\rm co}\times A^{\rm co}} Q_{\alpha, \beta},
\end{align*}
where we have used that the diagonalization and symmetrization of $B_1\times B_2 \cup B_2\times B_1$ for any $B_1, B_2\subset \R^m$ is given by $(B_1\cap B_2)\times (B_1\cap B_2)$.
\end{Remark} 

We continue with a lemma that will be used later on in Section~\ref{subsec:lsc_relaxation} to give a characterization of the sublevel sets of $\widehat W^{\rm slc}$. 
\begin{Lemma}\label{lem:intersectionsKj}
For $j\in \mathbb N$, let $K_j\subset \R\times \R$ be compact, symmetric and diagonal. If the sets $K_j$ are nested, i.e.~$K_j\supset K_{j+1}$ for all $j\in \mathbb N$, then 
\begin{align*}
\bigcap_{j\in \mathbb N} K_j^{\rm sc} = \big(\bigcap_{j\in \mathbb N}K_j\big)^{\rm sc}. 
\end{align*}
\end{Lemma}

\begin{proof} 
One inclusion follows directly from the definition of separately convex hulls. For the other one, let $(\xi, \zeta)\in \bigcap_{j\in \mathbb N}K_j^{\rm sc}$. Then for each $j\in \mathbb N$, there exists according to~\eqref{eq5} an element $(\alpha_j, \beta_j)\in K_j$ 
with $(\xi, \zeta)\in Q_{\alpha_j, \beta_j}$,
 and therefore 
\begin{align}\label{eq77}
(\xi, \zeta) = t_js_j(\alpha_j, \alpha_j) +  t_j(1-s_j) (\alpha_j, \beta_j) + s_j(1-t_j) (\beta_j, \alpha_j)+  (1-t_j)(1-s_j)(\beta_j, \beta_j)
\end{align} 
with $s_j, t_j\in [0,1]$. 
By compactness, we know that after passing to subsequences, we can assume that $s_j\to s\in [0,1]$, $t_j\to t\in [0,1]$, and $(\alpha_j, \beta_j)\to (\alpha, \beta)\in \bigcap_{j\in \mathbb N} K_j$ as $j\to \infty$. Finally, taking $j\to \infty$ in~\eqref{eq77} shows that $(\xi, \zeta)\in Q_{\alpha, \beta} \subset (\bigcap_{j\in \mathbb N} K_j)^{\rm sc}$.
\end{proof}

Inspired by the definition of extreme points in the separately convex sense, see e.g.~\cite[Definition~7.30]{Dac08}, we introduce here directional extreme points for subsets of $\R^m\times \R^m$. These can be used to refine the characterization formula~\eqref{eq5}, see Corollary~\ref{cor:square} below.

\begin{Definition}
Let $E\subset \R^m\times \R^m$ be separately convex. Then $(\xi, \zeta)\in E$ is a directional extreme point if the identity $(\xi, \zeta) = t (\xi_1, \zeta_1) + (1-t)(\xi_2, \zeta_2)$ for any $t\in (0,1)$ and any $(\xi_1, \zeta_1), (\xi_2, \zeta_2)\in E$ with $\xi_1=\xi_2$ or $\zeta_1=\zeta_2$ implies that $\xi=\xi_1=\xi_2$ and $\zeta = \zeta_1=\zeta_2$.  

For general $E\subset \R^m\times \R^m$, we say that $(\xi, \zeta)\in \R^m\times \R^m$ is a directional extreme point if $(\xi,  \zeta)$ is a directional extreme point for $E^{\rm sc}$ in the above mentioned sense. 

We denote the set of all directional extreme points of a set $E$ by $E_{\rm dex}$.
\end{Definition}

\begin{Remark}\label{rem:extremepoints}
If $m=1$,~\cite[Proposition~7.31]{Dac08} shows that $E_{\rm dex}\subset E$. The argument can be directly extended to the vectorial setting $m>1$, exploiting~\eqref{separatelyconvex_hull} and~\eqref{construction_sc}.
\end{Remark}

The representation formula~\eqref{eq5} can be simplified by considering only unions of squares whose vertices are directional extreme points of $E$.
\begin{Corollary}\label{cor:square}
	Let $E\subset \R\times \R$ be symmetric and diagonal. Then
	\begin{align}\label{eq6}
E^{\rm sc}  = \bigcup_{(\alpha, \beta)\in E_{\rm dex}} Q_{\alpha, \beta}.
	\end{align}
\end{Corollary}

\begin{proof} 
It suffices to show that for any $(\alpha, \beta)\in E\setminus E_{\rm dex}$, there exists a point $(\tilde \alpha, \tilde \beta)\in E$ different from $(\alpha, \beta)$ such that  $Q_{\alpha, \beta}\subset Q_{\tilde \alpha, \tilde \beta}$. The statement follows then in view of~\eqref{eq5}. 

Let $(\alpha, \beta)\in E\setminus E_{\rm dex}$. Then, in particular, $(\alpha,\beta)\in E^{\rm sc}$, so that $(\alpha, \beta)\in Q_{\tilde \alpha, \tilde \beta}$ for some $(\tilde \alpha, \tilde \beta)\in E$ according to \eqref{eq5}. In other words, there are $(\tilde \alpha, \tilde \beta) \in E$ and $t, s \in [0,1]$ such that 
\begin{align*}
(\alpha, \beta) = t(\tilde \alpha, s\tilde \alpha + (1-s)\tilde \beta)+ (1-t)(\tilde \beta, s\tilde \alpha+(1-s)\tilde \beta),
\end{align*} 
cf.~Remark~\ref{rem:failure}\,a). 
Since $(\alpha, \beta)$ is not an extreme point for $E$, we can suppose that $(\tilde \alpha, \tilde\beta)\neq (\alpha, \beta)$.  Finally, the observation that $Q_{\alpha, \beta}\subset Q_{\tilde \alpha, \tilde \beta}$ concludes the proof.
\end{proof} 

We close this section with a representation of separately convex hulls in terms of measures. For $K\subset \R^m\times \R^m$ non-empty and compact, one obtains the following alternative characterization of $K^{\rm sc}$, which is essentially a reformulation of~\eqref{separatelyconvex_hull} and~\eqref{construction_sc}: 
\begin{align*}
K^{\rm sc} =\bigcup_{i=0}^\infty \{[\Lambda]: \Lambda\in \mathcal M^{\rm sc}_i(K)\} 
\end{align*}
where 
 $\mathcal M^{\rm sc}_0(K) : = \{\delta_{(\xi, \zeta)}: (\xi, \zeta)\in K\}$ and for $i\in \mathbb{N}$,
\begin{align*}
\mathcal M^{\rm sc}_i(K) := \bigl\{\lambda \Lambda_1 + (1-\lambda)\Lambda_2 &: \Lambda_1, \Lambda_2\in \mathcal M^{\rm sc}_{i-1}(K), \lambda\in [0,1],\\  & \qquad [\Lambda_1 - \Lambda_2] \in \{(0, \xi), (\xi, 0):\xi\in \R^m\}\bigr\},
\end{align*}
In general, the measures whose barycenters yield elements in $K^{\rm sc}$ cannot be expected to be of product form. If $m=1$, however, this is the case, as the next lemma shows. 

\begin{Lemma}\label{lem:KscYM}
Let $K\subset \R\times\R$ be non-empty, symmetric, diagonal, and compact. Then,
\begin{align*}
K^{\rm sc} = \{[\Lambda]: \Lambda=\nu\otimes \mu,\ \nu, \mu\in \mathcal Pr(\R), \ {\rm supp\,}\Lambda\subset K\}.
\end{align*}
\end{Lemma}
\begin{proof} 
One inclusion is a simple consequence of Corollary~\ref{cor:square}. Indeed, if $(\xi, \zeta)\in K^{\rm sc}$, then by~\eqref{eq6} there is $(\alpha, \beta)\in K_{\rm dex} \subset K$ such that $(\xi, \zeta)\in Q_{\alpha, \beta}$. We choose $t, s\in [0,1]$ such that $\xi=t \alpha + (1-t)\beta$ and $\zeta=s\alpha +(1-s) \beta$, and set $\nu=t\delta_\alpha + (1-t)\delta_\beta\in \mathcal Pr (\R)$ and $\eta=s\delta_\alpha + (1-s)\delta_\beta\in \mathcal Pr (\R)$. Then 
\begin{align*}
\Lambda  & = \nu\otimes \eta =st \delta_{(\alpha, \alpha)} + t(1-s) \delta_{(\alpha, \beta)}+ s(1-t)\delta_{(\beta, \alpha)} + (1-t)(1-s) \delta_{(\beta, \beta)}
\end{align*}
is a product measure supported in $\{\alpha, \beta\}\times \{\alpha, \beta\}\subset K$ such that $[\Lambda]= ([\nu], [\eta])=(\xi, \zeta)$. 

For the reverse implication, let $\Lambda=\nu\otimes \mu$ with $\nu, \mu\in \mathcal Pr(\R)$ such that ${\rm supp\,} \Lambda\subset K$. 
Since the characteristic function $\chi_{K^{\rm sc}}:\R\times \R \to [0, \infty]$ is lower semicontinuous due to the compactness of $K$, which again implies that $K^{\rm sc}$ is compact according to Remark~\ref{rem:Esc_compact}, it follows from Lemma~\ref{lem:sepJensen} that 
$$
\int_{\mathbb R}\int_{\mathbb R} \chi_{K^{\rm sc}}(\xi,\zeta)\,d\nu(\xi)\, d\mu(\eta)\geq \chi_{K^{\rm sc}}([\nu], [\mu]).
$$
Recalling that $\chi_K\geq \chi_{K^{\rm sc}}$, the assumption that ${\rm supp\,}\Lambda\subset K$
yields $0\geq \chi_{K^{\rm sc}}([\nu], [\mu])$, or equivalently, $[\Lambda] = ([\nu], [\mu])\in K^{\rm sc}$, as stated. 
\end{proof} 

\begin{Remark}
	\label{remKscYMdimm}
	If $m>1$ and $K\subset \mathbb R^m\times \mathbb R^m$ is non-empty, symmetric, diagonal, and compact such that $K^{\rm sc}$ is also compact, and the structure condition
	\begin{equation}\label{structKschat}
\widehat{K^{\rm sc}}=\bigcup_{(\alpha, \beta)\in K} Q_{\alpha, \beta}
	\end{equation}
	 with cubes $Q_{\alpha, \beta}$ as defined in~\eqref{cubesQalphabeta} 
	holds, then analogous arguments to those in the proof of~the previous lemma allow us to derive that
	\begin{align*}
	\widehat{K^{\rm sc}} \subset \{[\Lambda]: \Lambda=\nu\otimes \mu,\ \nu, \mu\in \mathcal Pr(\R^m), \ {\rm supp\,}\Lambda\subset K\} \subset {K^{\rm sc}}.
	\end{align*}
\end{Remark}

\section{Nonlocal inclusions}\label{sec:nonlocal_inclusions}
For a set $E\subset \R^m\times \R^m$, we consider
\begin{align}\label{def:AK}
\mathcal{A}_E :=\{u\in L^\infty(\Omega;\R^m): v_u(x,y):=(u(x), u(y))\in E \text{ for 
a.e.~$(x,y)\in \Omega\times \Omega$}\}.
\end{align}

The main focus of this section is to prove the characterization result for the limits of weakly converging sequences in $\mathcal{A}_K$ with compact $K\subset \R\times \R$ stated in Theorem~\ref{theo:weakclosure_intro}. In the first subsection, we lay important groundwork by investigating the role of the set $E$ in $\mathcal A_E$. This gives important structural insight into the interplay between nonlocality effects and pointwise constraints, which are also interesting per se.

\subsection{Alternative representations of \boldmath{$\mathcal{A}_E$}}\label{subsec:alternative} The next result shows that the set $E\setminus \widehat E$ has no influence on the solutions to the nonlocal inclusion $(u(x), u(y))\in E$ for 
a.e.~$(x,y) \in \Omega\times \Omega$.

\begin{Proposition}\label{lem:Ik=ItildeK} 
Let $E, F\subset \R^m\times \R^m$ be closed. 
Then $\mathcal A_E= \mathcal A_F$ if and only if $\widehat{E}=\widehat{F}$. 

In particular, 
\begin{align}\label{AE=AEhat}
\mathcal A_{E} = \mathcal A_{\widehat E}.
\end{align} 
\end{Proposition}

\begin{proof} 
To show that equality of $\mathcal A_E$ and $\mathcal A_F$ implies that $\widehat E=\widehat F$, it suffices to prove that $\widehat E\subset \widehat F$. In fact, the reverse inclusion follows then from interchanging the roles of $E$ and $F$.
The case $\widehat E=\emptyset$ is trivial. Otherwise, let $(\xi, \zeta)\in \widehat E$, and consider the piecewise constant function
\begin{align*}
u(x) =\begin{cases} \xi &\text{for $x\in \Omega_\xi$,}\\ \zeta & \text{for $x\in \Omega_\zeta:=\Omega\setminus \Omega_\xi$,}
\end{cases} \qquad x\in \Omega,
\end{align*}
where $\Omega_\xi\subset \Omega$ is measurable with $\mathcal{L}^n(\Omega_\xi)>0$ and $\mathcal{L}^n(\Omega\setminus \Omega_\xi)>0$. By definition, $u\in L^\infty(\Omega;\R^m)$, and since $(\xi, \zeta)\in \widehat E\subset E$, it holds that also $(\xi, \xi), (\zeta,\zeta), (\zeta, \xi)\in E$. Hence, $u\in \mathcal A_{E} = \mathcal A_F$, and therefore $(\zeta, \xi), (\xi, \zeta), (\xi, \xi), (\zeta, \zeta)\in F$. This shows $(\xi, \zeta)\in \widehat F$. 

Notice that the converse implication, i.e.~$\mathcal A_E=\mathcal A_F$ if $\widehat E=\widehat F$, follows immediately, if one knows~\eqref{AE=AEhat}. To prove the latter, we start by observing that  $\mathcal A_E=\mathcal A_{E^{\rm sym}}$. Indeed, if $u\in \mathcal A_E$, then also $u\in \mathcal A_{E^T}$, and therefore $u\in \mathcal A_{E^{\rm sym}}$, because $E^{\rm sym}=E\cap E^T$. 
Thus, from now we assume $E$ to be symmetric.

Next, we will show that a specific class of subsets of $E$ can be removed without affecting $\mathcal{A}_E$. Precisely, if
$B\subset \R^m\times \mathbb R^m$ is such that 
 \begin{align}\label{condition_C}
 [\pi_{1}(B)\times \pi_{1}(B)]\cap E= \emptyset \quad \text{or}\quad [\pi_{2}(B)\times \pi_{2}(B)]\cap E= \emptyset,
 \end{align}
then 
\begin{align}\label{AK=AKC}
\mathcal{A}_{E} = \mathcal{A}_{E\setminus B}.
\end{align} 

To see this, let $B\subset \R^m\times \R^m$ satisfy the first condition in~\eqref{condition_C} (the reasoning in case the second condition holds is analogous), and consider $u\in \mathcal{A}_E$, assuming to the contrary that $u\notin \mathcal{A}_{E\setminus B}$.
	 Then there exists an $(\mathcal{L}^n\otimes \mathcal{L}^n)$-measurable set $N\subset \Omega\times \Omega$ with positive measure such that $(u(x), u(y))\in B$ for all $(x,y)\in N$. 
	By Tonelli's theorem or Cavalieri's principle, there exists $\bar{y}\in \Omega$ with $\mathcal{L}^n(\mathfrak{N}^{\bar y}_1)>0$; recall that $\mathfrak{N}^{\bar y}_1$ stands for the section in the first variable of $N$ at $\bar{y}$, cf.~Subsection~\ref{not}. Hence, 
	\begin{align*}
	(u(x), u(\bar y)) \in B \quad \text{for all $x\in \mathfrak{N}^{\bar y}_1$,} 
	\end{align*} 
	or equivalently, using projections, 
	$u(x)\in \pi_1(B)$ for $x\in \mathfrak{N}_1^{\bar y}$.
	This leads to 
	\begin{align*}
	(u(x), u(y))\in \pi_{1}(B)\times \pi_{1}(B) \quad \text{for all $(x, y)\in \mathfrak{N}^{\bar y}_1 \times \mathfrak{N}^{\bar y}_1$. }\end{align*}
	In view of~\eqref{condition_C}, we infer that $(u(x), u(y))\notin E$ for $(x, y)\in \mathfrak{N}^{\bar y}_1 \times \mathfrak{N}^{\bar y}_1$, which contradicts the assumption that $u\in \mathcal A_E$, and concludes the proof of~\eqref{AK=AKC}.

Next we apply~\eqref{AK=AKC} to suitable sets whose union amounts to $E\setminus \widehat E$. 
Owing to the fact that the complement $E^c$ of $E$ in $\R^m\times \R^m$ is open, one can find for any vector of rational numbers $\xi\in \mathbb{Q}^m$ with $(\xi, \xi)\notin E$ an open cube $]\alpha_\xi, \beta_\xi[\times]\alpha_{\xi}, \beta_{\xi}[\subset E^c$ with $\alpha_\xi, \beta_\xi\in \R^m$ such that $\xi\in ]\alpha_\xi, \beta_\xi[$. 

For each such $\xi$, one can apply~\eqref{AK=AKC} with the two choices $B=\R^m\times ]\alpha_\xi, \beta_\xi[$ and $B=]\alpha_\xi, \beta_\xi[\times \R^m$
to deduce that 
\begin{align}\label{AE=AEBU}
\mathcal A_E = \mathcal A_{E\setminus B_\cup} \quad \text{with } B_\cup:= \bigcup_{\xi\in \mathbb{Q}^m, (\xi, \xi)\notin E} (\R^m\times ]\alpha_\xi, \beta_\xi[) \cup (]\alpha_\xi, \beta_\xi[ \times \R^m) .
\end{align}  
To see this, let $(\xi_i)_{i\in \mathbb{N}}$ be an enumeration of $\{\xi\in \mathbb{Q}^{m}: (\xi, \xi)\notin E\}$ and set
\begin{align*}
B_\cup^k :=  \bigcup_{i=1}^k (\R^m\times ]\alpha_{\xi_i}, \beta_{\xi_i}[) \cup (]\alpha_{\xi_i}, \beta_{\xi_i}[ \times \R^m) \quad \text{for $k\in \mathbb{N}$. }
\end{align*} 
Then,  
~\eqref{AE=AEBU} follows from the line of identities
\begin{align*}
\displaystyle\mathcal A = \bigcap_{k\in \mathbb{N}} \mathcal{A}_{E\setminus B_\cup^k} = \mathcal{A}_{\cap_{k\in \mathbb{N}}E\setminus B_{\cup}^k} = \mathcal{A}_{E\setminus \cup_{k\in \mathbb{N}}B_\cup^k} = \mathcal{A}_{E\setminus B_\cup},
\end{align*}
where the first equality results from an iterative application of~\eqref{AK=AKC} to $]\alpha_{\xi_i}, \beta_{\xi_i}[\times \R$ and $\R\times ]\alpha_i, \beta_i[$ for $i=1, \ldots, k$, leading to $\mathcal{A}=\mathcal{A}_{E\setminus B_\cup^k}$ for any $k\in \mathbb{N}$. 
While the second identity is a consequence of Lemma~\ref{lem:tool} below, the third identity is due to basic properties of unions and intersections of sets, and the last step makes use of the fact that $B_\cup=\bigcup_{k\in \mathbb{N}} B_\cup^k$ by construction.

Finally, accounting for~\eqref{tildeK_construction} along with the observation that $B_\cup=B_E$ yields that $E\setminus B_{\cup} = \widehat E$. In view of~\eqref{AE=AEBU}, this concludes the proof of~\eqref{AE=AEhat}. 
\end{proof}

\begin{Lemma}\label{lem:tool}
Let $\{E_k\}_{k\in \mathbb{N}}$ be a family of sets in $\R^m\times \R^m$. Then,
\begin{align*}
\bigcap_{k\in\mathbb{N}}\mathcal{A}_{E_k} = \mathcal{A}_{\cap_{k\in \mathbb{N}}E_k}. 
\end{align*} 
\end{Lemma}
\begin{proof}  
If $u\in \bigcap_{k\in \mathbb{N}} \mathcal{A}_{E_k}$, one can find for every $k\in \mathbb{N}$ a set $N_k\subset \R^m\times \R^m$ of zero $\mathcal{L}^{2m}$-measure such that $(u(x), u(y))\in E_k$ for all $(x,y)\in \R^m\times \R^m\setminus N_k$. With $N:=\bigcup_{k=1}^\infty N_k$, we have a set of vanishing measure with the property that
every 
$(x,y)\in \R^m\times \R^m\setminus N$ satisfies
\begin{align*}
(u(x), u(y))\in \bigcap_{k\in \mathbb{N}} E_k,
\end{align*} 
meaning that $u\in \mathcal{A}_{\cap_{k\in \mathbb{N}}E_k}$. This proves $\bigcap_{k\in\mathbb{N}}\mathcal{A}_{E_k} \subset \mathcal{A}_{\cap_{k\in \mathbb{N}}E_k}$. The other implication is trivial. 
\end{proof}

\begin{Remark} \label{rem:AK=AKtilde}
If $E\subset \R^m \times \R^m$ is not closed, the identity $\mathcal A_E=\mathcal A_{\widehat E} $ is in general not true. To see this, let $n=m$ and $\Omega=(0,1)^m$, and consider 
\begin{align*}
E=[0,1]^m \times [0,1]^m\setminus \{(\xi, \xi): \xi \in \R^{m}\}.
\end{align*}
Then, $\widehat E=\emptyset$, and hence, $\mathcal A_{\widehat E}=\emptyset$. On the other hand, the identity map $u(x)= x$ for $x\in\Omega$ satisfies $(u(x), u(y)) = (x,y)\in E$ for all $(x, y)\in \Omega\times \Omega\setminus \{(x,x):x\in \Omega\}$. 
Since the diagonal $\{(\xi, \xi):\xi\in\R^m\}$ has zero Lebesgue-measure in $\R^m$, $u\in \mathcal A_E$. 
\end{Remark}

The next lemma is the basis for a useful approximation result, which is formulated below in Corollary~\ref{cor:approx}. For shorter notation, we write $S^\infty(\Omega;\R^m)$ for the subspace of $L^\infty(\Omega;\R^m)$ of simple functions, i.e., $u\in S^\infty(\Omega;\R^m)$ if
\begin{align}\label{simple}
u(x) = \sum_{i=1}^k \mathbbm{1}_{\Omega^{(i)}}\xi^{(i)}, \quad x\in \Omega,
\end{align}
with $\{\Omega^{(i)}\}_{i=1,\ldots, k}$ a partition of $\Omega$ into $\mathcal L^n$-measurable sets and $\xi^{(i)}\in \R^m$ for $i=1, \ldots, k$. By possibly choosing a different representative, one may assume without loss of generality that $\mathcal L^n(\Omega^{(i)})>0$ for all $i=1, \ldots, k$.

\begin{Lemma}\label{lem:approx}
Let $E\subset \R^m\times \R^m$ be symmetric and diagonal. 
Then, for every $u\in \mathcal A_E$ there exists a sequence $(u_j)_j\subset\mathcal A_E\cap S^\infty(\Omega;\R^m)$ with $u_j\to u$ in $L^\infty(\Omega;\R^m)$.
\end{Lemma}

\begin{proof} 
The proof follows along the lines of standard arguments for approximating unconstrained bounded functions uniformly by simple ones. Yet, particular care is needed here when choosing the function values to guarantee that the nonlocal inclusion defining $\mathcal A_E$ is not violated. This last step critically exploits the assumption that $E=\widehat E$. 
For clarification regarding notations throughout this proof, we refer the reader to~Subsection~\ref{not}.

After choosing a suitable representative of $u\in \mathcal A_E$, we may assume that $\underline{z}\leq u(x)\leq \overline{z}$ for all $x\in \Omega$ with $\underline{z}, \overline{z}\in \R^m$. 
For $j\in \mathbb{N}$, we partition the set $[\underline{z}_1, \overline{z}_1[\times \dots \times[\underline{z}_m,\overline{z}_m[$ into $k$ half-open cuboids $Q_j^{(i)}\subset \R^m$ 
such that 
 \begin{align}\label{est123}
 {\rm diam\ } Q_j^{(i)} < \frac{1}{j}  \quad \text{ for all $i=1, \ldots, k$,}
 \end{align}
and define the $\mathcal L^n$-measurable sets 
\begin{align*}
\Omega_j^{(i)}=u^{-1}(Q_j^{(i)}) 
\end{align*} 
for $i=1, \ldots, k$. Then, $\bigcup_{i=1}^k \Omega_j^{(i)} = \Omega$.  Let $I_j\subset\{1, \ldots, k\}$ be the index set defined by
\begin{align}\label{>0}
\mathcal{L}^n(\Omega_j^{(i)})> 0 \qquad \text{for $i\in I_j$.}
\end{align} 
Possibly after rearranging, one may assume without loss of generality that $I_j=\{1, \ldots, l\}$ for some $l\in \mathbb{N}$ with $l\leq k$.

Consider the simple function
\begin{align}\label{simpleuk}
u_j(x) = \sum_{i=1}^l \mathbbm{1}_{\Omega^{(i)}_{j}}(x)u(x^{(i)}_j), \qquad x\in \Omega,
\end{align} 
where $x^{(i)}_j$ 
are constructed iteratively as described in the following. 
Setting 
\begin{align*}
M=\{(x,y)\in \Omega\times \Omega: (u(x), u(y))\in E\},
\end{align*} 
we observe that the symmetry and diagonality of $E$ carry over to $M$, that is, if $(x,y)\in M$, then also $(y, x), (x,x), (y,y)\in M$.  With the  notations for sections of $M$, 
let
\begin{align*}
N=\{x\in \Omega: \mathcal{L}^n(\mathfrak M^x) = \mathcal{L}^n(\Omega) \}. 
\end{align*}
Since $(\mathcal{L}^{n}\otimes \mathcal{L}^n)(\Omega\times \Omega)=(\mathcal{L}^{n}\otimes \mathcal{L}^n)(M) = \int_{\Omega} \mathcal{L}^n(\mathfrak M^x) \,dx$ 
and thus, $\mathcal{L}^n(\mathfrak M^x)=\mathcal{L}^n(\Omega)$ for $\mathcal{L}^n$-a.e.~$x\in \Omega$, it follows that 
\begin{align}\label{equal_measures}
\mathcal{L}^n(N) = \mathcal{L}^n(\Omega).
\end{align} 

Now, let $x^{(1)}_j\in \Omega_j^{(1)}\cap N$ (this set is indeed non-empty by~\eqref{equal_measures} and~\eqref{>0}) and iteratively for $i=2, \ldots, l$, 
\begin{align}\label{inclusion2}
x_{j}^{(i)} \in \Omega_{j}^{(i)}\cap N\cap \Bigl(\bigcap_{p=1}^{i-1}\mathfrak M^{x_j^{(p)}}\Bigr). 
\end{align}
Notice that the set on the right-hand side in~\eqref{inclusion2} has positive $\mathcal L^n$-measure and is therefore in particular not empty. 
Indeed, this follows from~\eqref{equal_measures} and~\eqref{>0} in combination with $\mathcal{L}^n\bigl(\bigcap_{p=1}^{i-1}\mathfrak M^{x_j^{(p)}}\bigr)=\mathcal{L}^{n}(\Omega)$ for all $i=2, \ldots, l$. The latter is a consequence of $x_j^{(p)}\in N$ for $p=1, \ldots, i-1$. 
By construction, $u(x_j^{(i)}) \in Q_j^{(i)}$ for $i=1, \ldots, l$, 
and
\begin{align*}
(x_{j}^{(i)}, x_{j}^{(i')})\in M\qquad \text{ for $i,i'=1, \ldots, l$}.
\end{align*}
In view of~\eqref{simpleuk}, it holds therefore that 
 \begin{align*}
 (u_j(x), u_j(y))\in \bigcup_{i, i'\in \{1, \ldots, p\}} \{(u(x_j^{(i)}), u(x_j^{(i')}))\} \subset E \quad \text{for $(\mathcal L^n\otimes \mathcal L^n)$-a.e.~ $(x,y)\in \Omega\times \Omega$,}
 \end{align*}
which implies that $u_j\in \mathcal A_E$ for any $j\in \mathbb N$.
Moreover, together with~\eqref{est123},
\begin{align*}
|u(x) - u_j(x)| <  \frac{1}{j}\qquad\text{ for $\mathcal{L}^n$-a.e.~$x\in \Omega$, }
\end{align*}
so that $u_j\to u$ in $L^\infty(\Omega;\R^m)$ as $j\to \infty$.
This shows that $(u_j)_j$ is an approximating sequence for $u$ with the stated properties.
\end{proof}

The following density statement for $\mathcal A_E$ with a closed set $E$ is an immediate consequence of Lemma~\ref{lem:approx} and Proposition~\ref{lem:Ik=ItildeK}.

\begin{Corollary}\label{cor:approx}
Let $E\subset \R^m\times \R^m$ be closed. 
Then $\mathcal A_E$ coincides with the closure of $\mathcal A_E\cap S^\infty(\Omega;\R^m)$ in $L^\infty(\Omega;\R^m)$.
\end{Corollary}

Based on this approximation result and the special properties of simple functions in $\mathcal A_E$, there is another way to represent $\mathcal A_E$, namely in terms of Cartesian products (cf. Definition~\ref{def:Cartesian_subset}).

\begin{Proposition}\label{prop:3}
If $E\subset \R^m\times \R^m$ is closed, then 
\begin{align}\label{repres2}
\mathcal A_E= \bigcup_{P\in \mathcal P_E}\mathcal A_P. 
\end{align}
\end{Proposition}

\begin{proof}
For the proof of the nontrivial inclusion, consider any $u\in \mathcal A_E$. We will show that there exists $A\subset \R^m$ with $A\times A\subset E$ such that $u\in \mathcal A_{A\times A}$. Then, $A\times A\subset P$ for some $P\in \mathcal P_E$, and therefore $u\in \mathcal A_P$.

First, we observe that~\eqref{repres2} holds for simple functions. In fact, if $u\in \mathcal S^\infty(\Omega;\R^m)\cap \mathcal A_E$, then it is of the form~\eqref{simple} with $(\xi^{(i)}, \xi^{(i')})\in E$ for all $i, i'=1, \ldots, k$. Here we use in particular that the sets $\Omega^{(i)}$ can be chosen to have positive $\mathcal L^n$-measure. Consequently,
\begin{align*}
v_u(\Omega\times \Omega)= u(\Omega)\times u(\Omega) = \bigcup_{i,i'=1}^k u(\Omega^{(i)}) \times u(\Omega^{(i')}) = \bigcup_{i,i'=1}^k \{(\xi^{(i)},\xi^{(i')})\}\subset E,
\end{align*}
which yields the statement in the case when $u$ is simple.

To prove~\eqref{repres2} in the general case, let $(u_j)_j$ be an approximating sequence resulting from Lemma~\ref{lem:approx}, so that 
\begin{align}\label{approx_4}
u_j\to u \quad \text{in $L^\infty(\Omega;\R^m)$.}
\end{align}
Due to the uniform boundedness of $(u_j)_j$ in $L^\infty(\Omega;\R^m)$, we may assume without loss of generality that $E$ is bounded, and hence compact. 
Since each $u_j$ is simple, one can thus find for every $j\in \mathbb N$ a compact set $A_j\subset \R^m$ with $P_j:=A_j\times A_j\subset E$ such that $u_j\in \mathcal A_{P_j}$. 

Next, we exploit the fact that the metric space of closed subsets of a compact set in $\R^m$ endowed with the Hausdorff distance $d^m_H$ in \eqref{Hausdorff} is compact, see~e.g.~\cite{Rog70} or \cite[Theorem~6.1]{AFP00} for Blaschke selection theorem.  
Hence, there is a subsequence of $(A_j)_j$ (not relabelled) and $A\subset \R^m$ compact 
such that $d_H^m(A_j, A)\to 0$ as $j\to \infty$.
In light of the relation 
\begin{align*}
d_H^{2m}(B\times B, D\times D) \leq 2\, d_H^m(B, D)
\end{align*}
for non-empty sets $B, D\subset \R^m$, this implies  that
\begin{align}\label{eq76}
d_H^{2m}(P_j, A\times A)=d_H^{2m}(A_j\times A_j, A\times A) \to 0\quad \text{as $j\to \infty$,}
\end{align}
and since $P_j\subset E$ for all $j\in \mathbb{N}$, it follows that $A\times A\subset E$.

Moreover, by~\eqref{approx_4} in combination with dominated convergence and~\eqref{eq76}, 
\begin{align*}
\int_\Omega\int_\Omega {\rm dist}(v_u, A\times A) \, dx\, dy & = \lim_{j\to \infty} \int_\Omega\int_\Omega {\rm dist}(v_{u_j}, A\times A) \, dx\, dy \\ &\leq \lim_{j\to \infty} \int_\Omega\int_\Omega {\rm dist}(v_{u_j}, P_j) \, dx\, dy + \lim_{j\to \infty} d_H^{2m}(P_j, A\times A) \mathcal L^n(\Omega)^2= 0.
\end{align*}
Hence, $v_u\in A\times A$ 
a.e.~in $\Omega\times \Omega$ or $u\in \mathcal A_{A\times A}$, which finishes the proof. 
\end{proof}
\begin{Remark}
Note that Proposition~\ref{prop:3} fails if $E$ is not closed. For the example in Remark~\ref{rem:AK=AKtilde}, it holds that $\mathcal{P}_E=\emptyset$, whereas $\mathcal{A}_E\neq \emptyset$.  
\end{Remark}

\subsection{Asymptotic analysis of sequences in \boldmath{$\mathcal A_K$}}\label{subsec:asymptotics}
For a compact set $K\subset \R^m\times \R^m$, in view of Remark \ref{rem1.2}\;a), we denote the $L^\infty$-weak$^\ast$ closure of $\mathcal A_K$ by $\mathcal A_K^\infty$, that is, 
\begin{align}\label{AinftyE}
\mathcal{A}_{K}^\infty := \{u\in L^\infty(\Omega;\R^m): u_j\weaklystar u \text{ in $L^\infty(\Omega;\R^m)$, } (u_j)_j\subset \mathcal{A}_K\}.
\end{align}
This section contains the proof of 
Theorem~\ref{theo:weakclosure_intro}, which 
can be reformulated in terms of~\eqref{AinftyE} as \begin{align}\label{theo:short}
\mathcal{A}_K^\infty=\mathcal{A}_{\widehat K^{\rm sc}}.
\end{align} 

We start with an auxiliary result showing that~the implication $\mathcal A_{\widehat K^{\rm sc}}\subset \mathcal A_{K}^\infty$ is true whenever $K$ consists of the vertices of a symmetric cube in $\R^m\times \R^m$. 

\begin{Lemma}\label{lem:1}
Let $\alpha, \beta\in \R^m$ and $K =\{\alpha, \beta\}\times \{\alpha, \beta\}$. Then 
\begin{align*}
 \mathcal{A}_{Q_{\alpha, \beta}} \subset \mathcal{A}_K^\infty,
\end{align*} 
recalling that $Q_{\alpha, \beta}=[\alpha, \beta]\times [\alpha, \beta]$, cf.~\eqref{cubesQalphabeta}. 
\end{Lemma}

\begin{proof} 
Suppose first that $u\in \mathcal A_{Q_{\alpha, \beta}}\cap S^\infty(\Omega;\R^m)$ and 
let $u$ as in~\eqref{simple} with $\mathcal L^n(\Omega^{(i)})>0$ for $i=1, \ldots, k$. 
Then, $\xi^{(i)}\in [\alpha, \beta]\subset \R^m$ for all $i=1, \ldots, k$, and there are $\lambda_i\in [0,1]$ such that $\xi^{(i)}=\lambda_i\alpha + (1-\lambda_i)\beta$. Moreover, let $Y_{\xi^{(i)}}\subset\, ]0,1[^n$ be measurable with $\mathcal L^n(Y_{\xi^{(i)}}) =\lambda_i$ and define $h^{(i)}$ as the $]0,1[^n$-periodic function given by 
\begin{align*}
h^{(i)}(y)=\begin{cases} \alpha & \text{for $y\in Y_{\xi^{(i)}}$,}\\ \beta & \text{for $]0,1[^n\setminus Y_{\xi^{(i)}}$,}
\end{cases} \qquad y\in ]0,1[^n.
\end{align*}
Setting
\begin{align*}
u_j(x) = \sum_{i=1}^k h^{(i)}(jx)\mathbbm{1}_{\Omega^{(i)}}(x)
\end{align*}
for $x\in \Omega$ and $j\in \mathbb N$, leads to $u_j\weaklystar u$ in $L^\infty(\Omega;\R^m)$ according to the Riemann-Lebesgue lemma on weak convergence of periodically oscillating sequences.  By construction, $(u_j(x), u_j(y))\in  \{\alpha, \beta\}\times \{\alpha, \beta\}=K$ for all $(x, y)\in \Omega\times \Omega$, so that $u_j\in \mathcal A_K$ for every $j\in \mathbb N$.  

For general functions $u\in \mathcal A_{Q_{\alpha, \beta}}$, we argue via approximation.
Let  $(\tilde{u}_k)_k\subset \mathcal A_{Q_{\alpha, \beta}}\cap S^\infty(\Omega;\R^m)$ be a sequence of simple functions such that $\tilde{u}_k\to u$ in $L^\infty(\Omega;\R^m)$ as $k\to \infty$, see Lemma~\ref{lem:approx}.The previous construction allows us to find for each $k\in \mathbb N$ a sequence $(\tilde u_{k,j})_j\subset \mathcal A_K$ with $\tilde u_{k,j}\weaklystar \tilde{u}_k$ in $L^\infty(\Omega;\R^m)$ as $j\to \infty$. By a version of Attouch's diagonalization lemma \cite[Lemma~1.15, Corollary~1.16]{Att84} (exploiting in particular that $L^\infty(\Omega;\R^m)$ is the dual of a separable space), we can select $k(j)\to \infty$ as $j\to \infty$ such that for $u_j:=\tilde u_{k(j),j}\in \mathcal A_K$,
\begin{align*}
u_j\weaklystar u \quad \text{in $L^\infty(\Omega;\R^m)$.}
\end{align*} 
This shows that $u\in \mathcal A_K^\infty$ and completes the proof.
\end{proof}

\begin{proof}[Proo\textbf{}f of Theorem~\ref{theo:weakclosure_intro}]
We prove separately the two inclusions that make up~\eqref{theo:short}. 

	First, let $u\in \mathcal A_K^\infty$. Then, in view of Proposition~\ref{lem:Ik=ItildeK}, there exists a sequence $(u_j)_j\subset L^\infty(\Omega;\R^m)$ with $v_{u_j} \in \widehat{K}$ a.e.~in $\Omega\times \Omega$ such that $u_j\weaklystar u$ in $L^\infty(\Omega;\R^m)$. Moreover, let $\{\nu_x\otimes \nu_y\}_{(x, y)\in \Omega\times \Omega}$ be the Young measure generated by $(v_{u_j})_j$, cf.~Lemma~\ref{asprop2.3Pedregal}. Since $K$, and hence also $\widehat K$, is compact, so is $\widehat K^{\rm sc}$ in the case $m=1$ according to Remark~\ref{rem:Esc_compact}. For $m>1$, the compactness of $\widehat K^{\rm sc}$ is guaranteed directly by assumption. 
As a result, the map 
\begin{align*}
	\R^m\times \R^m\to [0, \infty),\quad (\xi, \zeta)\mapsto {\rm dist}^2((\xi, \zeta), \widehat K)
	\end{align*}
	 is lower semicontinuous, and we infer from Theorem \ref{FTYM} that 
	\begin{align*}
	0 & = \lim_{j\to \infty}\int_{\Omega}\int_{\Omega} {\rm dist}^2\big(v_{u_j}, \widehat K\big) \,dx\, dy \geq
\lim_{j\to \infty}\int_{\Omega}\int_\Omega\int_{\mathbb R^m}\int_{\mathbb R^m} {\rm dist}^2((\xi,\zeta), \widehat K)\, d \nu_x(\xi)\otimes\nu_y(\zeta)\,dx\, d y 
\geq 0. 
	\end{align*}	
	Hence, $\nu_x\otimes \nu_y$ is supported in $\widehat K\subset \widehat K^{\rm sc}$ for a.e.~$(x,y)\in \Omega\times \Omega$. 
	By Lemma \ref{lem:sepJensen} applied with $W=\chi_{\widehat K^{\rm sc}}$, it follows then that $(u(x),u(y))=([\nu_x], [\nu_y])\in \widehat K^{\rm sc}$ for a.e.~$(x,y)\in \Omega \times \Omega$, and
	thus, $u\in \mathcal A_{\widehat K^{\rm sc}}$. 

To prove the reverse inclusion, recall that the second assumption on $\widehat{K}^{\rm sc}$ in the case $m>1$ 
says that
\begin{align}\label{ass76}
\widehat{{\widehat K}^{\rm sc}}=\bigcup_{(\alpha, \beta)\in \widehat K}Q_{\alpha, \beta} \color{olive} \quad \text{with $\mathcal{P}_{\widehat{K}^{\rm sc}}\subset \{Q_{\alpha, \beta}:(\alpha,\beta)\in \widehat K\}$.}
\end{align}

Now, we combine Lemma~\ref{lem:square} if $m=1$, or the previous assumption~\eqref{ass76} if $m>1$, with Proposition~\ref{prop:3} and Lemma~\ref{lem:1} to infer that 
\begin{align}\label{finaleq}
\mathcal A_{\widehat K^{\rm sc}} = 
\color{olive} \mathcal{A}_{\bigcup_{P\in \mathcal P_{ \widehat K^{\rm sc}}}P} = \bigcup_{P\in \mathcal P_{\widehat K^{\rm sc}}} \mathcal{A}_P \subset \color{black} \bigcup_{(\alpha, \beta)\in \widehat K} \mathcal A_{Q_{\alpha, \beta}} \color{black} \subset \bigcup_{(\alpha, \beta)\in \widehat K} \mathcal A^\infty_{\{\alpha, \beta\}\times \{\alpha, \beta\}} \subset \mathcal A_K^\infty.
\end{align}
This finishes the proof.
 \end{proof}

\begin{Remark}
	\label{observation}
a) 	If $m=1$, one could replace $\widehat{K}$ in the second, third and fourth term in~\eqref{finaleq} by $\widehat{K}_{\rm dex}$, 
	simply using Lemma~\ref{lem:square} instead of Corollary~\ref{cor:square}, and taking into account that ${\widehat K}_{\rm dex}\subset \widehat K$ by Remark~\ref{rem:extremepoints}.
	
	b) For examples of sets satisfying~\eqref{ass76} see Remarks \ref{rem:hat1}\,b) and \ref{rem:failure}\,c). 
 \end{Remark}

The following result is an immediate consequence of Theorem~\ref{theo:weakclosure_intro} in conjunction with Proposition \ref{lem:Ik=ItildeK} and Remark~\ref{rem:failure}\,a), cf. also Remark \ref{rem1.2}\,a).

\begin{Corollary}\label{cor:Khatsc}
Let $K$ as in Theorem~\ref{theo:weakclosure_intro}. Then 
$\mathcal{A}_K$ is $L^\infty$-weakly$^\ast$ 
closed if and only if 
\begin{align}\label{con1}
\widehat{\widehat{K}^{\rm sc}} = \widehat K.
\end{align}
For $m=1$, the condition~\eqref{con1} is equivalent with the separate level convexity of $\widehat K$. 
\end{Corollary}
%


\subsection{Characterization of Young measures generated by sequences in \boldmath{$\mathcal A_K$}}\label{subsec:YM_characterization}

For $K\subset \R^m\times \R^m$ compact, let $\mathcal Y_K^\infty$ be the set of Young measures generated by a sequence of nonlocal vector fields associated with $(u_j)_j\subset \mathcal A_K$; more precisely, 
\begin{align}
\begin{array}{l}\label{def_YKinfty}
\mathcal Y_K^\infty :=\{\Lambda\in L_w^\infty(\Omega\times \Omega;\mathcal Pr(\R^m\times \R^m)): v_{u_j} \stackrel{YM}{\longrightarrow} \Lambda \text{ with $(u_j)_j\subset \mathcal A_K$}\}.
\end{array}
\end{align}
Regarding barycenters, we observe that
\begin{align}\label{barycenter}
\{[\Lambda]=\langle \Lambda, {\rm id} \rangle: \Lambda\in \mathcal Y_K^\infty\}=\{v_u: u\in \mathcal A_K^\infty\}\subset L^\infty(\Omega\times \Omega;\R^m\times \R^m).
\end{align}

As a consequence of 
Proposition~\ref{lem:Ik=ItildeK}, 
Lemma~\ref{asprop2.3Pedregal} and Theorem~\ref{FTYM}\,(iii), 
\begin{align}\label{YM_inclusion}
\mathcal Y_K^\infty = \mathcal Y_{\widehat K}^\infty \subset \widetilde{\mathcal Y}^\infty_{\widehat K} =\mathcal Y_{\widehat K}, 
\end{align}
where for any compact $C\subset \R^m\times \R^m$,
\begin{align*}
\begin{array}{l}
\mathcal Y_C  := \{\Lambda \in L_w^\infty(\Omega\times \Omega; \mathcal Pr(\R^m\times \R^m)): \Lambda_{(x,y)}=\nu_x\otimes \nu_y \text{ with $\nu\in L_w^\infty(\Omega;\mathcal Pr(\R^m))$ and }  \\[0.2cm] 
\hspace{6.7cm} {\rm supp\,} \Lambda_{(x,y)} \subset C \text{ for a.e.~$(x,y)\in \Omega\times\Omega$}\},
\end{array}
\end{align*} 
and 
$\widetilde{\mathcal Y}_{C}^\infty$ is a modification of $\mathcal Y_C^\infty$ in the sense that the exact inclusion is weakened to an approximate version, i.e.,
\begin{align*}
\begin{array}{l}
\widetilde{\mathcal Y}_{C}^\infty  :=\{\Lambda\in L_w^\infty(\Omega\times \Omega;\mathcal Pr(\R^m\times \R^m)):  v_{u_j} \stackrel{YM}{\longrightarrow} \Lambda \text{ with $(u_j)_j\subset L^\infty(\Omega;\R^m)$ such that} \\[0.2cm] \hspace{6.8cm}\text{${\rm dist}(v_{u_j}, C)\to 0$ in measure as $j\to \infty$}\}.
\end{array}
\end{align*}

In the simple special case, when $K$ has the form of a Cartesian product (then clearly, $K=\widehat K$), we are able to show that equality holds in~\eqref{YM_inclusion}. The proof combines well-known results from the theory of Young measures with a projection argument. 
Note that for more general $K$ the projection result fails due to non-trivial interactions between the different variables.

\begin{Proposition}\label{prop:YMchar_Cartesian}
Let $K\subset \R^m\times \R^m$ such that $K=A\times A$ with $A\subset \R^m$ compact.
Then, 
\begin{align*}
\mathcal Y_K^\infty = \mathcal Y_K.
\end{align*}
\end{Proposition}

\begin{proof}
In view of~\eqref{YM_inclusion}, it remains to show that $\widetilde{\mathcal Y}_{K}^\infty \subset \mathcal Y_K^\infty$. To this end, we project the sequences generating the Young measures in $\widetilde{\mathcal Y}_{K}^\infty$ onto $K$. 

Let $\Lambda\in \widetilde{\mathcal Y}_{K}^\infty$ be generated by $(v_{\tilde u_j})_j$ with $(\tilde u_j)_j\subset L^\infty(\Omega;\R^m)$ such that ${\rm dist}(v_{\tilde u_j}, K) = {\rm dist}(v_{\tilde u_j}, A\times A)\to 0$ in measure as $j\to \infty$. 
By measurable selection~\cite[Section~6.1.1, Theorem~6.10]{FoL07}, one can find a measurable and essentially bounded function $u_j:\Omega \to \R^m$ with 
\begin{align*}
u_j(x) \in {\rm argmin\;}_{\xi\in A} {\rm dist}(\tilde u_j(x), \xi) \quad \text{for a.e.~$x\in \Omega$.}
\end{align*}
Then by construction, $v_{u_j}\in A\times A=K$ a.e.~in $\Omega\times \Omega$, and $v_{u_j}-v_{\tilde u_j}\to 0$ in measure as $j\to \infty$. The latter implies in particular that $(v_{u_j})_j$ generates the same Young measure as $(v_{\tilde u_j})_j$, namely $\Lambda$. Hence, $\Lambda\in \mathcal Y_K^\infty$. 
\end{proof}

With these prerequisites at hand, we can derive the following characterization of Young measures generated by sequences with nonlocal constraints.

\begin{Theorem}\label{theo:YMcharacterization}
Let $K\subset \R^m\times \R^m$ be compact. Then $\mathcal Y_K^\infty = \bigcup_{P\in \mathcal P_{\widehat K}} \mathcal Y_P$. 
\end{Theorem}

\begin{proof} 
Owing to the fact that any set in $\mathcal P_K$ is a subset of $K$ with the form of a Cartesian product in $\R^m\times \R^m$, the inclusion $ \bigcup_{P\in \mathcal P_{\widehat K}} \mathcal Y_P \subset \mathcal Y_K^\infty$ follows immediately from Proposition~\ref{prop:YMchar_Cartesian}. 

For the proof of reverse inclusion, consider $(v_{u_j})_j$ as in~\eqref{def_YKinfty}, generating the Young measure $\Lambda\in \mathcal Y_K^\infty$.
Then, Proposition~\ref{prop:3} implies for every $j\in \mathbb N$ the existence of  \color{olive} $A_j\subset \R^m$ \color{black} compact such that 
\begin{align*}
v_{u_j}\in P_j :=A_j\times A_j\subset \mathcal P_K = \mathcal P_{\widehat K} \quad \text{ a.e.~in $\Omega\times \Omega$.} 
\end{align*}
Arguing similarly to Proposition~\ref{prop:3},  we conclude (possibly after passing to a non-relabelled subsequence of $(A_j)_j$) that $d_H^m(A_j, A)\to 0$ as $j\to \infty$ for some $A\subset \R^m$ compact with the property that $A\times A \subset \widehat K$. 
It follows then in view of
 \begin{align*}
 {\rm dist}(v_{u_j}, A\times A)\leq {\rm dist}(v_{u_j}, P_j) + d_H^{2m}(P_j, A\times A) = d_H^{2m}(P_j, A\times A)\leq 2\,d_H^m(A_j, A)
 \end{align*}  
a.e.~in $\Omega\times \Omega$, 
that $\|{\rm dist}(v_{u_j}, A\times A)\|_{L^\infty(\Omega\times \Omega;\R^m\times \R^m)}\to 0$ as $j\to \infty$. Then, by the fundamental theorem of Young measures in Theorem~\ref{FTYM}\,(iii), ${\rm supp\,}\Lambda \subset A\times A\subset$ a.e.~in $\Omega\times \Omega$. If we take $P$ as the maximal Cartesian subset of $\widehat K$ containing $A\times A$, this shows that  
 $\Lambda\in \mathcal Y_P$ and finishes the proof.
\end{proof}

\begin{Remark} 
Based on Theorem~\ref{theo:YMcharacterization}, we can now give a short alternative proof of~\eqref{theo:short}. 
Precisely, combining~Theorem~\ref{theo:YMcharacterization} with ~\eqref{barycenter} and Lemma~\ref{asprop2.3Pedregal} shows that
\begin{align*}
\mathcal A_K^\infty & = \{[\nu]: \nu\in L_w^\infty(\Omega;\mathcal Pr(\R^m)), \Lambda_{(x,y)}= \nu_x\otimes \nu_y, \Lambda\in \mathcal Y^\infty_K \} \\ & = \{[\nu]: \nu\in L_w^\infty(\Omega;\mathcal Pr(\R^m)), \Lambda_{(x,y)}= \nu_x\otimes \nu_y, \Lambda \in \textstyle \bigcup_{P\in \mathcal P_{\widehat{K}}} \mathcal Y_P\} =: \widetilde{\mathcal{A}}.
\end{align*}

\color{olive} Since Lemma~\ref{lem:KscYM} (for $m=1$) and Remark~\ref{remKscYMdimm} (for $m>1$) imply that $\bigcup_{P\in \mathcal P_{\widehat K}} \mathcal A_{P^{\rm sc}} \subset \widetilde{\mathcal A} \subset \mathcal A_{\widehat K^{\rm sc}}$, and
\begin{align*}
\mathcal A_{\widehat K^{\rm sc}} =\bigcup_{(\alpha, \beta)\in \widehat K} \mathcal A_{Q_{\alpha, \beta}} = \bigcup_{\{\alpha, \beta\}\times \{\alpha, \beta\}\subset \widehat K} \mathcal A_{Q_{\alpha, \beta}} \subset \bigcup_{A\times A\in \mathcal{P}_{\widehat K}} \mathcal A_{A^{\rm co}\times A^{\rm co}} = \bigcup_{P\in \mathcal P_{\widehat K}} \mathcal A_{P^{\rm sc}}
\end{align*}
due to Lemma~\ref{lem:square} (for $m=1$) and~\eqref{ass76} (for $m>1$), the identity~\eqref{theo:short} follows. 
\end{Remark}

\section{Nonlocal indicator functionals}\label{sec:indicator}

The aim of this section is to relate the previous results with the theory of nonlocal unbounded functionals, in particular, with indicator functionals.

\subsection{Lower semicontinuity and relaxation}\label{subsec:51}
For $K\subset \R^m\times \R^m$, we define the indicator functional $I_K:L^\infty(\Omega;\R^m)\to \{0, \infty\}$  by
\begin{align}\label{indicatorfunctional}
I_K(u) := \int_{\Omega}\int_{\Omega} \chi_{K}(u(x), u(y))\,dx \,dy 
=\begin{cases}
0 & \text{if $u\in \mathcal{A}_K$,}\\
\infty & \text{otherwise;}
\end{cases} 
\end{align} 
recall the notations from~\eqref{indicator} and \eqref{def:AK}. 
 It is clear from the second equality in~\eqref{indicatorfunctional} that the lower semicontinuity and relaxation of $I_K$ regarding the weak$^\ast$ topology in $L^\infty(\Omega;\R^m)$ are closely related to the asymptotic behaviour of sequences in $\mathcal{A}_K$ with respect to the same topology, cf.~Remark \ref{rem1.2}\,a). 
In fact, the $L^\infty$-weak$^\ast$ lower semicontinuity of $I_K$ corresponds to the weak$^\ast$ closedness of ${\mathcal A}_K$, while determining its relaxation, i.e.,
\begin{align*}
I_K^{\rm rlx}(u)
:= \inf\{\liminf_{j\to \infty} I_{K}(u_j): u_j \weaklystar u \text{ in $L^\infty(\Omega;\R^m)$}\}
\end{align*}
for all $u\in L^\infty(\Omega;\R^m)$,
is equivalent to characterizing the $L^\infty$-weak$^\ast$ closure of $\mathcal{A}_K$, denoted by $\mathcal{A}_K^\infty$ in \eqref{AinftyE}. 

Formulated here again for the readers' convenience, the counterparts of Corollary~\ref{cor:Khatsc} and Theorem~\ref{theo:weakclosure_intro} in terms of indicator functionals are the following.

\begin{Corollary}\label{cor:chara_wlscp_indicator}
	Let $K\subset \R^m\times \R^m$ be as in Theorem~\ref{theo:weakclosure_intro}. 
	\begin{itemize}
\item[$(i)$]	The functional $I_K$ is $L^\infty$-weakly$^\ast$ lower semicontinuous, if and only if
\begin{align*}
\widehat{\widehat{K}^{\rm sc}} =\widehat K;
\end{align*}
for $m=1$, this is the same as $\widehat K$ (or equivalently,  $\chi_{\widehat K}$) being separately convex. 

\item[$(ii)$]	 Moreover, $I^{\rm rlx}_K=I_{\widehat{K}^{\rm sc}}$, where the latter is the functional in \eqref{indicatorfunctional} associated with the separately convex hull ${\widehat K}^{\rm sc}$.
\end{itemize}
\end{Corollary}

\subsection{Young measure relaxation}\label{sec:YMrelaxation}

As an application of Theorem \ref{theo:YMcharacterization}, we determine the relaxation in the Young measure setting of a class of extended-valued double-integral functionals. This result can be viewed as a generalization of \cite[Theorem 6.1]{BMC18}. 

For $K\subset \R^m\times \R^m$, let the functional $I^{\mathcal Y}_K: L^\infty_w
(\Omega;\mathcal Pr(\mathbb R^m))\to \{0,\infty\}$ be defined by
\begin{align}\label{IYK}
 I_{K}^{\mathcal{Y}}(\nu)  := \min_{P\in \mathcal{P}_{\widehat K}} \int_{\Omega}\int_{\Omega} \int_{\R^m}\int_{\R^m} \chi_P(\xi, \zeta) \,d\nu_x(\xi)\, d\nu_y(\zeta) \,dx\, dy  = \begin{cases} 0 & \text{if $\nu\otimes \nu\in \bigcup_{P\in \mathcal{P}_{\widehat K}} \mathcal Y_P$}, \\  \infty & \text{otherwise,}
	\end{cases} 
	\end{align}
for $\nu \in L_{w}^\infty(\Omega;{\mathcal Pr}(\mathbb R^m))$. 

The follwing reformulation of Theorem \ref{theo:YMcharacterization} states a Young measure relaxation for nonlocal indicator functionals in general dimensions.

\begin{Corollary}\label{cor:YMrelaxation}
Let $K\subset \R^m\times \R^m$ be compact.
\begin{itemize}
	\item[$(i)$] If the sequence $(u_j)_j \subset L^\infty(\Omega;\mathbb R^m)$ generates the Young measure $\nu$, in formulas, $u_j\stackrel{YM}{\longrightarrow} \nu$, then
	\begin{align}\label{lbIYKhat}
	\liminf_{j\to \infty} I_K(u_j)  \geq  I_{K}^{\mathcal{Y}}(\nu).
	\end{align}
	\item[$(ii)$] For every $\nu\in L_w^\infty(\Omega;{\mathcal Pr}(\R^m))$ there exists a sequence $(u_j)_j\subset L^\infty(\Omega;\R^m)$ with $u_j\stackrel{YM}{\longrightarrow} \nu$ such that 
	\begin{align*}
	\lim_{j\to \infty} I_K(u_j)  = I_{K}^{\mathcal{Y}}(\nu).
	\end{align*}
	\end{itemize}
\end{Corollary}

\begin{Remark}
	\label{rem:IKhatsc=IYKhat} 
	
	If $K\subset \mathbb R^m \times \mathbb R^m$ is compact as in Theorem~\ref{theo:weakclosure_intro}, i.e.~$\widehat K^{\rm sc}$ is compact and satisfies~\eqref{ass76}, we can directly verify the expected relations between the functionals arising from classical and Young measure relaxation of $I_K$.
	For any 
	$\nu \in L^\infty_w(\Omega;{\mathcal Pr}(\mathbb R^m))$,
	\begin{align}\label{lbIfirstmoment}
	I^{\mathcal Y}_{K}(\nu)\geq I_{{\widehat K}^{\rm sc}}([\nu]);
	\end{align}
	moreover, for every $u \in L^\infty(\Omega;\R^m)$, there exists a Young measure $\nu \in L^\infty_w(\Omega;{\mathcal Pr}(\mathbb R^m))$ with $[\nu]=u$ such that 
	\begin{align}
	\label{ubIfirstmoment}
	I^{\mathcal Y}_{K}(\nu)\leq I_{{\widehat K}^{\rm sc}}([\nu])=I_{{\widehat K}^{\rm sc}}(u).
	\end{align}
 To see~\eqref{ubIfirstmoment}, it is 
		enough to invoke  
		 Theorem \ref{theo:weakclosure_intro} and the characterizion in Theorem \ref{theo:YMcharacterization}.
		
	As regards the justification of~\eqref{lbIfirstmoment}, we may assume without loss of generality that $I^{\mathcal Y}_{K}(\nu)=0$; thus, there exists $P=A\times A \in {\mathcal P}_{\widehat{K}}$ with $A\subset \R^m$ such that $\nu_x \otimes \nu_y \in P$ for a.e.~$(x,y)\in \Omega \times \Omega$. By Theorem \ref{theo:YMcharacterization}, one can find a sequence $(u_j)_j \subset {\mathcal A}_P$ generating $\nu$ and converging weakly$^\ast$ to $u= [\nu]$ in $L^\infty(\Omega;\R^m)$, with $u\in A^{\rm co}$ for a.e.~in $\Omega$. These observations, together with Lemma~\ref{lem:sepJensen} and $A^{\rm co}\times A^{\rm co} = (A\times A)^{\rm sc}\subset {\widehat K}^{\rm sc}$, imply that 
	\begin{align*}
	I^{\mathcal Y}_{K}(\nu) &\geq \int_{\Omega}\int_{\Omega} \int_{\R^m}\int_{\R^m} \chi_{A^{\rm co}\times A^{\rm co}}(\xi, \zeta) \,d\nu_x(\xi)\, d\nu_y(\zeta) \,dx\, dy \\
	&\geq	\int_{\Omega}\int_{\Omega} \int_{\R^m}\int_{\R^m} \chi_{{\widehat K}^{\rm sc}}(\xi, \zeta) \,d\nu_x(\xi)\, d\nu_y(\zeta) \,dx\, dy \\
	&	\geq\int_\Omega \int_\Omega \chi_{{\widehat K}^{\rm sc}}([\nu_x], [\nu_y])\,d x \,d y 
	=I_{\widehat K^{\rm sc}}([\nu]),
	\end{align*} 
	as stated. 
\end{Remark}

	As a consequence of Corollary \ref{cor:YMrelaxation} and the results in \cite[Section~6]{BMC18}, one can deduce a Young measure representation for the relaxation of constrained nonlocal integral functionals of the type
	\begin{align}\label{extended}
	L^\infty(\Omega;\mathbb R^m) \ni u \to \int_\Omega\int_\Omega w((x,y,u(x),u(y))\,dx \,d
	y+I_K(u),
	\end{align}
	where $w:\Omega \times \Omega \times \mathbb R^m \times \mathbb R^m\to \mathbb R_\infty$ is exactly as in \cite[Theorem 6.1]{BMC18}.
	Indeed, the superadditivity of $\liminf$, \eqref{lbIYKhat}, and \cite[Theorem 6.1]{BMC18} entail for every sequence $(u_j)_j \subset L^\infty(\Omega;\mathbb R^m)$ with $u_j\stackrel{YM}{\longrightarrow} \nu$ that
	\begin{align*}
&	\liminf_{j\to \infty} \Bigl(\int_\Omega\int_\Omega w(x,y,u_j(x),u_j(y))\,dx\,dy+I_K(u_j)\Bigr) \\  &\qquad \qquad \geq 
	\int_\Omega \int_\Omega \int_{\mathbb R^m}\int_{\mathbb R^m}w(x,y,\xi,\zeta)d\nu_x(\xi)\,d \nu_y(\zeta)\,dx \,dy+ I_{K}^{\mathcal{Y}}(\nu).
	\end{align*}
	On the other hand, if $\nu\in L_w^\infty(\Omega;{\mathcal Pr}(\R^m))$, we choose $(u_j)_j$ to be a sequence as in~Corollary~\ref{cor:YMrelaxation}\,(ii), and 
	 apply the version of the fundamental theorem on Young measures in~\cite[Proposition~3.6]{BMC18} to conclude that
	\begin{align*}
&	\lim_{j\to \infty} \Bigl(\int_\Omega\int_\Omega w(x,y,u_j(x),u_j(y))\,dx \,d
	y+I_K(u_j)\Bigr) \\&  \qquad = \int_\Omega \int_\Omega \int_{\mathbb R^m}\int_{\mathbb R^m}w(x,y,\xi,\zeta)\,d\nu_x(\xi)\,d \nu_y(\zeta)\,dx \,dy+ I_{K}^{\mathcal{Y}}(\nu).
	\end{align*}

\subsection{Notions of nonlocal convexity}\label{subsec:nonlocal_convexity} In \cite{BMC18} and the references therein, the authors introduce and analyze different notions of nonlocal convexity for inhomogeneous finite-valued double-integral functionals, including nonlocal convexity, nonlocal convexity for Young measures, and a nonlocal Jensen inequality.
Here, we transfer these notions to our context of homogeneous indicator functionals in the scalar setting, i.e.~functionals $I_K$ and $I_K^{\mathcal{Y}}$ as in~\eqref{indicatorfunctional} and~\eqref{IYK} with $K$ as in Theorem~\ref{theo:weakclosure_intro}, and discuss their relation.  

Let us first define 
the condition referred to as  nonlocal convexity \eqref{NC}: 
For every  $w \in L^\infty(\Omega;\mathbb R^m),$ the function 
\begin{equation}\label{NC}
\tag{NC}
\iota_w: \mathbb R^m \to \{0,\infty\}, \quad \iota_w(\xi) :=\int_{\Omega}\chi_{\widehat K}(\xi,w(x))\,dx
\quad \hbox{ is convex.}
\end{equation}
A generalization of condition \eqref{NC} is the following  nonlocal convexity for Young measures \eqref{NY}, which requires that for every $\nu \in L^\infty_w
(\Omega;\mathcal Pr(\mathbb R^m))$, the function
\begin{equation}
\label{NY}\tag{NY}
\Im_{\nu}: \mathbb R^m \to \{0,\infty\}, \quad  \Im_{\nu}(\xi ) := 
\int_{\Omega}\int_{\mathbb R^m}  \chi_{\widehat K}(\xi,\zeta)  \,d\nu_x(\zeta)\,dx 
\quad \hbox { is convex.}
\end{equation}

Inspired by Pedregal~\cite[Proposition~3.1 and (4.3)]{Ped97}, we consider the nonlocal Jensen's inequality 
\begin{equation}\label{NJ}
\tag{NJ} 
I^{\mathcal Y}_K(\nu) \geq I_K([\nu])
\end{equation}
for any $\nu \in L^\infty_w(\Omega;\mathcal Pr(\mathbb R^m))$, cf.~\eqref{IYK} for the definition of $I_K^{\mathcal{Y}}$. 
Finally, we denote by ${\rm (SC)}$ the separate convexity of $\chi_{\widehat K}$ (or equivalently, of $\widehat K$).

The next proposition establishes the equivalence of all these notions. In particular,  in view of Corollary~\ref{cor:Khatsc} and Remark \ref{rem1.2}\;a), they are all necessary and sufficient for $L^\infty$-weak$^\ast$ lower semicontinuity of $I_K$.

\begin{Proposition}\label{equivNonlocalcond}
 If $K\subset \mathbb R^m \times \mathbb R^m$ is as in Theorem~\ref{theo:weakclosure_intro},  then 
	\begin{align*}
	\color{olive} \eqref{NJ}
	 \Leftrightarrow \color{black} {\rm (SC)}
	 \Leftrightarrow \eqref{NC} \Leftrightarrow\eqref{NY} .
	\end{align*}
\end{Proposition}

\begin{proof}[Proof] For the proof of ${\rm (NJ)}\Leftrightarrow {\rm (SC)}$, we make use of~\eqref{lbIfirstmoment} and \eqref{ubIfirstmoment}, together with the fact that $I_K= I_{K^{\rm sc}}$ implies 
\begin{align*}
	\widehat{K}^{\rm sc} = \widehat{\widehat{K}^{\rm sc}} = \widehat{K}
	\end{align*} 
due to Proposition \ref{lem:Ik=ItildeK} and Lemma~\ref{Khatsepconv}.

	The arguments behind the other implications are straight-forward.  
	The implication ${\rm (SC)}\Rightarrow {\rm (NY)}$ follows right from the definition of separate convexity of $\chi_{\widehat K}$.
	Via the identification of $u \in L^\infty(\Omega;\R^m)$ with the family of Dirac measures $\{\delta_{u(x)}\}_{x\in \Omega}$, the condition \eqref{NY} is clearly at least as strong as~\eqref{NC}.
	 To see ${\rm (NC)}\Rightarrow {\rm (SC)}$, it suffices to restrict \eqref{NC} to constant functions and exploit the symmetry of $\widehat{K}$. 
\end{proof}

\section{Nonlocal supremal functionals}\label{6}
The main focus of this section is the proof of Theorem \ref{theo:main}, which is based on the results established previously.
In what follows, $W:\mathbb R^m\times \mathbb R^m\to \mathbb R$ is always assumed to be lower semicontinuous and coercive. 
In terms of the level sets of $W$, this means that $L_c(W)$ are compact for any $c\in \R$.
	
We start, in view of Remark \ref{rem1.2}\;a), 
with a characterization result for $L^\infty$-weak$^\ast$ lower semicontinuity of functionals as in~\eqref{ourfunct}  that exploits the relations with nonlocal indicator functionals and nonlocal inclusions.
It is a nonlocal version of the analogous statement in the local setting pointed out first by Acerbi, Buttazzo \& Prinari~in \cite[Remark~4.4]{ABP02} and used later e.g.~by Briani, Garroni \& Prinari in~\cite[Proposition 4.4]{BGP04}, see~also \cite[Lemma 1.4]{BJW01}.

\begin{Proposition}
	\label{propequivsupremalunbounded}
Recalling the definitions in~\eqref{ourfunct}, \eqref{def:AK} and ~\eqref{indicatorfunctional}, the following three statements are equivalent:
	\begin{itemize}
	\item[$(i)$] $J$ is $L^\infty$-weakly$^\ast$ lower semicontinuous; 
	\item[$(ii)$] $\mathcal{A}_{L_c(W)}$ is $L^\infty$-weakly$^\ast$ closed for all $c\in \R$;
	\item[$(iii)$] $I_{L_c(W)}$ is $L^\infty$-weakly$^\ast$ lower semicontinuous for all $c\in \R$.
	\end{itemize}
\end{Proposition}

\begin{proof} The equivalence of $(ii)$ and $(iii)$ follows immediately from \eqref{indicatorfunctional}. It remains to prove that $(i)$ and $(ii)$ are equivalent. 
	
	Assuming that $(i)$ holds, consider any $c \in \mathbb R$ and any sequence $(u_j)_j\subset\mathcal A_{L_c(W)}$ and $u\in L^\infty(\Omega;\R^m)$ such that $u_j\rightharpoonup^\ast u$ in $L^\infty(\Omega;\mathbb R^m)$. Since the $L^\infty$-weak$^\ast$ lower semicontinuity of $J$ ensures that
	$$
	\displaystyle{\supess_{(x,y)\in\Omega \times \Omega}W(u(x),u(y))\leq \liminf_{j\to \infty}\displaystyle{\supess_{(x,y)\in\Omega \times \Omega}W(u_j(x),u_j(y))} \leq c,}
	$$
	we conclude that $(u(x), u(y))\in L_c(W)$ for a.e.~$(x,y)\in \Omega\times \Omega$, meaning $u \in \mathcal A_{L_c(W)}$. This proves $(ii)$.

	For the reverse implication, we take $u_j \rightharpoonup^\ast u$ in $L^\infty(\Omega;\R^m)$ with 
	\begin{align*}
	\lim_{j\to \infty} J(u_j)=\liminf_{j\to \infty} J(u_j) < \infty. 
	\end{align*}

	Let $C_{\rm sup}:=\supess_{(x,y)\in \Omega \times \Omega}W(u(x),u(y))$ and assume by contradiction that 
	$$\lim_{j\to \infty} J(u_j) = \lim_{j\to \infty}\supess_{(x,y)\in \Omega\times \Omega}W(u_j(x), u_j(y)) = c <C_{\rm sup}.$$ 
	Then, for any $\varepsilon \in (0, C_{\rm sup}-c)$ there exists an index $N=N(\varepsilon) \in \mathbb{N}$ such that for every $j \geq N$, 
	\begin{align*}
	 \supess_{(x,y)\in \Omega \times \Omega}W(u_j(x), u_j(y)) \leq c+ \varepsilon<C_{\rm sup},
	 \end{align*} 
	 or equivalently, $u_j \in \mathcal A_{L_{c+\varepsilon}(W)}$. Due to $(ii)$,
	  we infer that $u \in \mathcal A_{L_{c+\varepsilon}(W)}$, and hence, $W(u(x), u(y))\leq c+\varepsilon$ a.e.~in $\Omega\times \Omega$. The desired contradiction follows now from 
	\begin{align*}
	C_{\rm sup} = \supess_{(x,y)\in \Omega \times \Omega}W(u(x),u(y))	\leq c+\varepsilon < C_{\rm sup},
	\end{align*}	
	which concludes the proof. 
\end{proof}

\subsection{Lower semicontinuity and relaxation}\label{subsec:lsc_relaxation} 
The following characterization result, which can be obtained from combining Corollary~\ref{cor:Khatsc} and Proposition~\ref{propequivsupremalunbounded}, generalizes Theorem \ref{theo:main}\,(i) to the vectorial setting, cf.~Lemma~\ref{Khatsepconv}.

\begin{Corollary} \label{cor}
Let $J$ be a nonlocal supremal functional as in~\eqref{ourfunct} such that $\widehat{L_c(W)}$ is compact and satisfies~\eqref{ass76} for every $c\in \R$. 
Then, $J$ is $L^\infty$-weakly$^\ast$ lower semicontinuous if and only if for all $c\in \R$,
\begin{align*}
\widehat{\widehat{L_c(W)}^{\rm sc}}=\widehat{L_c(W)}.
\end{align*}
\end{Corollary}

\begin{Remark}\label{rem:firstparttheomain}
Notice that the sufficiency of the separate convexity of the symmetrized and diagonalized sublevel sets of $W$ to ensure $L^\infty$-weak$^\ast$ lower semicontinuity of $J$ as in~\eqref{ourfunct} holds without any further assumptions also in the vectorial case $m>1$. The argument employs Proposition~\ref{suffseplevconv} under consideration of \eqref{J=tildeJ} and~\eqref{levelset_W} below. 
\end{Remark}

Our next goal is to establish a representation formula for the relaxation of $J$. Inspired by the previous corollary, we define $\widehat{W}:\R^{m}\times \R^m\to \R$ by
\begin{align}\label{defWhat}
\widehat{W}(\xi, \zeta) := {\rm inf} \{c\in \R: (\xi, \zeta)\in \widehat{L_{c}(W)}\},\quad (\xi, \zeta)\in \R^m\times \R^m. 
\end{align} 
Then, for any $c\in \R$,
\begin{align}\label{levelset_W}
L_c(\widehat{W}) = \widehat{L_c(W)}.
\end{align}
Since the sublevel sets of $W$ are compact, this shows in particular that the level sets of $\widehat{W}$ are compact as well, and hence, that $\widehat{W}$ is lower semicontinuous. Moreover, $\widehat W$ is coercive due to $\widehat W\geq W$, and symmetric, i.e.,~$\widehat W(\xi,\zeta)= \widehat W(\zeta, \xi)$ for every $(\xi,\zeta)\in \mathbb R^m \times \mathbb R^m$, by definition, cf.~\eqref{hatK}. 

It is crucial to realize that a functional $J$ as in~\eqref{ourfunct} has a uniquely determined supremand $W$  only up to symmetrization and diagonalization in the sense of~\eqref{defWhat}. 
To be precise, it holds that
\begin{align}\label{J=tildeJ}
J(u)= \supess_{(x,y)\in \Omega\times \Omega} W(u(x), u(y)) =  \supess_{(x,y)\in \Omega\times \Omega} \widehat{W}(u(x), u(y)) =:\hat{J}(u) 
\end{align}
for $u\in L^{\infty}(\Omega;\mathbb R^m)$;
indeed, along with Proposition~\ref{lem:Ik=ItildeK} and~\eqref{levelset_W},
\begin{align}\label{representation_supremal}
 \supess_{(x,y)\in \Omega\times \Omega} \widehat{W}(u(x), u(y)) &= \inf\{ c\in \R: u\in \mathcal{A}_{L_c(\widehat{W})}\} 
 =  \inf\{ c\in \R: u\in \mathcal{A}_{\widehat{L_c(W)}}\} \\ &= \inf\{ c\in \R: u\in \mathcal{A}_{L_c(W)}\} = \supess_{(x,y)\in \Omega\times \Omega} W(u(x), u(y)).  \nonumber
 \end{align}
 
In light of  Definition~\ref{seplevconv} for the separate level convex envelope of a function and Definition~\ref{separateconvexitysets} for the separately convex hull of a set, it is immediate to  see that 
\begin{align}\label{slcenvelope}
L_c(\widehat W^{\rm slc}) \supset L_c(\widehat W)^{\rm sc} \qquad \text{for every $c \in \mathbb R$.}
\end{align}

If $m=1$, one can show that even  
equality holds in \eqref{slcenvelope}. In particular, if we recall the properties of $\widehat{W}$ and  Remark~\ref{rem:Esc_compact}, this implies that 
 $\widehat W^{\rm slc}:\mathbb R \times \mathbb R \to \mathbb R$ is lower semicontinuous and coercive. 

\begin{Lemma}
Let $\widehat W$ as in~\eqref{defWhat} and $m=1$. Then, for every $c\in \R$,
\begin{align}\label{slcenvelope=}
L_c(\widehat W^{\rm slc}) = L_c(\widehat W)^{\rm sc}.
\end{align}
\end{Lemma}

\begin{proof}
Define the auxiliary function
\begin{align*}
V(\xi, \zeta) := {\rm inf} \{c\in \R: (\xi, \zeta)\in L_{c}(\widehat W)^{\rm sc}\},\quad (\xi, \zeta)\in \R\times \R.
\end{align*} 
Since all sublevel sets of $\widehat W$ are compact, symmetric and diagonal, Lemma~\ref{lem:intersectionsKj} entails that for any $c\in \R$,
\begin{align*}
L_c(V) =  \bigcap_{j\in \mathbb N} L_{c+\frac{1}{j}}(\widehat W)^{\rm sc} = \Big(\bigcap_{j\in \mathbb N} L_{c+\frac{1}{j}}(\widehat W)\Big)^{\rm sc} =  L_c(\widehat W)^{\rm sc},  
\end{align*}
which shows that $V$ is separately level convex. 
Due to $\widehat W\geq V$, we conclude that $\widehat W^{\rm slc}\geq V$, and consequently $L_c(\widehat{W}^{\rm slc})\subset L_c(V) = L_c(\widehat W)^{\rm sc}$ for all $c\in \R$. Considering that the other inclusion is immediate in view of the definition of the separately level convex envelope $\widehat{W}^{\rm slc}$ completes the proof. 
\end{proof}

With these preparations, we can now prove Theorem \ref{theo:main}\,(ii), namely the relaxation result for supremal nonlocal functionals in the scalar case.
\begin{Proposition}\label{prop:relaxation2}
	Let $J$ be the functional in \eqref{ourfunct} with $m=1$.
The relaxation of $J$ given by its $L^\infty$-weak* lower semicontinuous envelope
\begin{align}\nonumber
J^{\rm rlx}(u)= \inf \{\liminf_{j\to \infty} J(u_j): u_j\weaklystar u \text{ in $L^\infty(\Omega)$}\}, \quad u\in L^\infty(\Omega),
\end{align}
admits the supremal representation
\begin{align*}
J^{\rm rlx}(u)= \supess_{(x,y)\in\Omega\times \Omega} \widehat{W}^{\rm slc}(u(x), u(y)),
\qquad u\in L^{\infty}(\Omega).
\end{align*}
\end{Proposition}
\begin{proof}
The argument for the lower bound on $J^{\rm rlx}$ relies on Corollary~\ref{cor} and  \eqref{J=tildeJ}, together with the simple observation that $\widehat {W}\geq \widehat W^{\rm slc}$. 

For the upper bound on $J^{\rm rlx}$, take any $u\in L^\infty(\Omega)$ such that 
\begin{align*}
c:=\supess_{(x,y)\in \Omega\times \Omega} \widehat{W}^{\rm slc}(u(x), u(y))<\infty.
\end{align*} 
Then there exists a sequence of real numbers $(c_k)_k$ with $c_k\searrow c$ as $k\to \infty$ such that owing to~\eqref{representation_supremal} and~\eqref{levelset_W},
\begin{align*}
u \in \mathcal A_{L_{c_k}(\widehat{W}^{\rm slc})}= \mathcal A_{L_{c_k}(\widehat{W})^{\rm sc}} = \mathcal A_{\widehat{L_{c_k}(W)}^{\rm sc}} 
\quad \text{ for all $k\in \mathbb{N}$.}
\end{align*}

Now, Theorem \ref{theo:weakclosure_intro} applied to $\mathcal A_{\widehat{L_{c_k}(W)}^{\rm sc}}$  for every $k\in \mathbb N$ guarantees the existence of a sequences $(u_{k, j})_j\subset \mathcal A_{L_{c_k}(W)}$
 with $u_{k,j}\weaklystar u$ in $L^\infty(\Omega)$ as $j\to \infty$. 
Via diagonalization (see~\cite[Lemma~1.15, Corollary~1.16]{Att84}), one can select a diverging subsequence $k(j)\to \infty$ as $j\to \infty$ such that the sequence $(u_j)_j$ with $u_j:=u_{k(j), j} \in \mathcal A_{L_{c_{k(j)}}(W)}$ for $j\in \mathbb N$ satisfies 
$u_j \weaklystar u$ in $L^\infty(\Omega)$.

Then, 
\begin{align*}
J^{\rm rlx}(u)\leq \limsup_{j\to \infty} J(u_j)
\leq \limsup_{j\to \infty} c_{k(j)} = c = \supess_{(x,y)\in \Omega\times\Omega} \widehat{W}^{\rm slc}(u(x), u(y)).\end{align*}
\end{proof}

Under additional assumptions, we can generalize  
Proposition~\ref{prop:relaxation2} to the vectorial case.
\begin{Remark}\label{rem:Jwlsc_m}
Let $W:\R^m\times \R^m\to \R$ with $m>1$ such that for any $c\in \R$, the sublevel set $\widehat{L_c(W)}$ is compact and satisfies both~\eqref{ass76} and \eqref{slcenvelope=}. 
Then, the $L^\infty$-weak$^\ast$ lower semicontinuous envelope of $J$ is then given by the nonlocal supremal functional with density 
 $\widehat{\widehat{W}^{\rm slc}}$, which may in general be different from $\widehat{W}^{\rm slc}$, as~Remark~\ref{rem:hat1}\,b) indicates. 
\end{Remark}

\subsection{Explicit examples of lower semicontinuous functionals and relaxations}\label{subsec:examples}
To illustrate the general results of Section~\ref{subsec:lsc_relaxation}, we present a few examples of nonlocal $L^\infty$-functionals whose supremands have multiwell structure. 

In the scalar setting, we determine explicit relaxation formulas for two nonlocal four-well supremands. Even though the sets of wells can be transformed into each other via rotation and scaling, their relaxations feature qualitative differences. 

\begin{Example}\label{ex} 
Throughout this example, $|\cdot |_{\square}$ stands for the maximum norm on $\mathbb{R}\times \mathbb{R}\cong \R^2$, i.e.~$|(\xi,\zeta)|_\square=\max\{|\xi|, |\zeta|\}$ for $\xi, \zeta\in \mathbb{R}$, and we write $B_r^\square(\xi, \zeta)$ to denote the corresponding closed balls of radius $r>0$ with center in $(\xi, \zeta)\in \R\times \R$. Moreover, ${\rm dist}_\square(\cdot, E)$ indicates the maximum distance from a set $E\subset \R\times \R$, cf.~Section~\ref{not} for the corresponding notations with respect to the Euclidean norm. 

a) Let $J$ as in~\eqref{ourfunct} with $W(\xi, \zeta) = {\rm dist}((\xi, \zeta), K_6)$ for $(\xi, \zeta)\in \R\times \R$, where $K_6=\{-1,1\}\times \{-1,1\}$ is the compact, diagonal and symmetric set from~\eqref{K4K5}. Then, for $c\geq 0$, the level sets of $W$ are unions of balls, precisely, $L_c(W) = \bigcup_{(\xi, \zeta)\in K_6} B_c(\xi, \zeta)$, while $L_c(W)=\emptyset$ for $c<0$. It follows along with~\eqref{levelset_W} that for $c\geq 0$, 
\begin{align*}
L_c(\widehat{W})=\widehat{L_c(W)}= \bigcup_{(\xi, \zeta)\in K_6}B_{\frac{c}{\sqrt{2}}}^\square(\xi, \zeta),
\end{align*}
which is the union of the maximal squares contained in the balls whose union gives $L_c(W)$, and hence,
$\widehat{W}(\xi, \zeta) = \sqrt{2} \,{\rm dist}_{\square}((\xi,\zeta), K_6)$ for $(\xi, \zeta)\in \R\times \R$.

Due to~\eqref{slcenvelope}, $ L_c(\widehat{W}^{\rm slc})=L_c(\widehat{W})^{\rm sc}= B^\square_{1+\frac{c}{\sqrt{2}}}(0,0)$ for $c\geq 0$, 
and we infer that
\begin{align*}
 \widehat{W}^{\rm slc}(\xi, \zeta) =\sqrt{2}\,\max\big\{|(\xi, \zeta)|_\square - 1, 0\big\} 
\end{align*}
for $(\xi, \zeta)\in \R\times\R$. 
By Proposition~\ref{prop:relaxation2}, this gives rise to an explicit expression for $J^{\rm rlx}$.

A curiosity related to the nonlocal behavior of $W$ and the associated necessary diagonalization is that, unlike for local supremal functionals, $\widehat W^{\rm slc}$ is not everywhere smaller than $W$; 
for instance, $\widehat W^{\rm slc}(1, 1+r) = \sqrt{2}r > r = W(1, 1+r)$ for any $r>0$.

b) Consider $J$ from~\eqref{ourfunct} with $W(\xi, \zeta) = {\rm dist}((\xi, \zeta), K_5)$ for $(\xi, \zeta)\in \R\times \R$ and the compact set $K_5=\{(0,1), (1,0), (0, -1), (-1, 0)\}$ from~\eqref{K4K5}. Similarly to a), the sublevel sets $L_c(W)$ are non-empty for $c\geq 0$, with $L_c(W)=\bigcup_{(\xi, \zeta)\in K_5} B_c(\xi, \zeta)$. We observe that  $L_c(\widehat{W}) = \widehat{L_c(W)}=\emptyset$ for $c<\frac{1}{\sqrt{2}}$, 
while for $c\geq\frac{1}{\sqrt{2}}$, a simple geometric argument shows that
\begin{align*}
L_c(\widehat{W})= \bigcup_{r\in [r_-(c), r_+(c)]} \partial B_{r}^\square(0,0)
\end{align*}
with $r_{\pm}(c)=\frac{1}{2}\max\{1\pm\sqrt{2c^2-1}, 0\}$, and consequently, $L_c(\widehat{W})^{\rm sc}= B_{r_+(c)}^\square(0,0)$. In view of~\eqref{slcenvelope}, we finally obtain
\begin{align*}
\widehat{W}^{\rm slc}(\xi, \zeta) =\begin{cases} \sqrt{\frac{1}{2}(2 |(\xi, \zeta)|_\square-1)^2 +\frac{1}{2}} & \text{for $|(\xi, \zeta)|_\square \geq \frac{1}{2}$,}\\ \frac{1}{\sqrt{2}} & \text{otherwise,}\end{cases}
\end{align*}
for $(\xi, \zeta)\in \R\times\R$, which yields an explicit formula for the relaxation $J^{\rm rlx}$, see Proposition~\ref{prop:relaxation2}.

 We point out that in this example, even the minimum of $W$ is smaller than that of $\widehat W^{\rm slc}$, precisely, 
$\min W = 0 < \tfrac{1}{\sqrt{2}}= \min \widehat W= \min \widehat W^{\rm slc}$. 
\end{Example}

The next examples show the $L^\infty$-weak$^\ast$ lower semicontinuity of two types of supremal functionals with symmetric two-well supremands in the vectorial setting. 

\begin{Example} \color{olive} Let $m>1$. \color{black}

a) For $K=\{(-\alpha, -\alpha), (\alpha, \alpha)\}\subset \R^m\times \R^m$ with $\alpha\in \R^m\setminus\{0\}$, let
$W(\xi, \zeta) = {\rm dist}_\square((\xi, \zeta), K):= \min_{\beta\in\{-\alpha, \alpha\}} \max\{|\xi-\beta|, |\zeta-\beta|\}$  for $(\xi, \zeta) \in \R^m\times \R^m$.  
Then the level sets for any $c\in \R$ are given by 
\begin{align*}
L_c(W) = \bigl(B_c(\alpha)\times B_c(\alpha)\bigr)\cup \bigl(B_c(-\alpha)\times B_c(-\alpha)\bigr),
\end{align*}
recalling that $B_{r}(\xi)=\{\zeta\in \R^m: |\zeta-\xi|\leq r\}$ for $r>0$ and $\xi\in \R^m$, cf.~Section~\ref{not}.  Note that $W$ is not separately level convex, since $L_c(W)$ fails to be separately convex for $c\geq |\alpha|$; in particular, Proposition~\ref{suffseplevconv} is not applicable here.  
However, as the union of 
Cartesian products of convex sets, all level sets of $W$ are clearly symmetric and diagonal, meaning $W=\widehat{W}$, and we can infer in light of Remark~\ref{rem:hat1}\,b) and~\eqref{levelset_W} that 
\begin{align*}
\widehat{\widehat{L_c(W)}^{\rm sc}} = \widehat{L_c(W)^{\rm sc}}= L_c(W) = \widehat{L_c(W)}. 
\end{align*}
By 
\color{olive} Corollary~\ref{cor}, \color{black} this condition is sufficient for  $L^\infty$-weakly$^\ast$ lower semicontinuity for $J$ as in~\eqref{ourfunct}. 

b) The same statement as in a) holds for $J$, if we use $K=\{(\alpha, -\alpha), (-\alpha, \alpha)\}$ with $\alpha\in \R^m\setminus\{0\}$ and set $W(\xi, \zeta) ={\rm dist}_\square((\xi, \zeta),K):= \min\{\max\{|\xi-\alpha|, |\zeta+\alpha|\}, \max\{|\xi+\alpha|, |\zeta-\alpha|\}\}$ for $(\xi, \zeta)\in \R^m\times \R^m$. Then, 
\begin{align*}
L_c(W) = \bigl(B_c(\alpha)\times B_c(-\alpha)\bigr) \cup \bigl(B_c(-\alpha)\times B_c(\alpha)\bigr)
\end{align*}
for $c\in \R$, and 
\begin{align*}
\widehat{L_c(W)} = \begin{cases} \bigl(B_{c}(\alpha)\cap B_c(-\alpha)\bigr) \times \bigl(B_c(\alpha) \cap B_c(-\alpha) \bigr) & \text{for $c\geq|\alpha|,$}\\  \emptyset & \text{otherwise.} \end{cases}
\end{align*}
Considering that these sets are already separately convex, we conclude again with \color{olive} Corollary~\ref{cor}. 
\end{Example}

\section*{Acknowledgements}
The authors would like to thank Giuliano Gargiulo and Martin Kru\v{z}\'ik for interesting discussions. CK was partially supported by a Westerdijk Fellowship from Utrecht University and by the NWO grant TOP2.17.012. EZ is a member of the Gruppo Nazionale per l'Analisi Matematica, la Probabilit\'a e le loro Applicazioni (GNAMPA) of the Istituto Nazionale di Alta Matematica (INdAM). 
This paper was written during visits of the authors at Mathematical Department of Utrecht University and at Dipartimento di Ingegneria Industriale dell' Universit\'a di Salerno, whose kind
hospitality and support are gratefully acknowledged. 
In particular, the support of GNAMPA through the program 'Professori Visitatori 2018' is gratefully acknowledged.


\bibliographystyle{abbrv}
\bibliography{NonlocalSup}

	\end{document}